\documentclass[12pt,reqno,sumlimits]{amsart}

 \usepackage{amssymb,amscd,amsmath,epsfig}
\usepackage{
  array,
   booktabs,
   dcolumn,
   rotating,
   shortvrb,
   tabularx,
   units,
   url,
 }
\newtheorem{thm}{Theorem}[section]
\newtheorem{cor}[thm]{Corollary}
\newtheorem{lem}[thm]{Lemma}
\newtheorem{prop}[thm]{Proposition}
\theoremstyle{definition}

\newtheorem{rem}{Remark}[section]

\newcommand{\R}{{\mathbb R}}

\newcommand{\Q}{{\mathbb Q}}
\newcommand{\C}{{\mathbb C}}

\newcommand{\Z}{{\mathbb Z}}

\newcommand{\calE}{{\mathcal E}}

\newcommand{\calG}{{\mathcal G}}
\newcommand{\calH}{{\mathcal H}}

\newcommand{\calR}{{\mathcal R}}

\renewcommand{\to}{\longrightarrow}

\newcommand{\ev}{\operatorname{ev}}

\newcommand{\ad}{{\operatorname{ad\,}}}

\newcommand{\weight}{(1+\Delta)^s}
\newcommand{\weightinv}{(1+\Delta)^{-s}}

\newcommand{\con}[2]{\nabla^{#1}_{#2}}

\newsavebox{\savepar}

\newcommand{\ip}[1]{\langle #1 \rangle}

\numberwithin{equation}{section}

\newcounter{labelflag} 
\newcommand{\labelon}{\setcounter{labelflag}{1}}
\newcommand{\Label}[1]{
                       \ifnum\thelabelflag=1
                          \ifmmode
                             \makebox[0in][l]{\qquad\fbox{\rm#1}}
                          \else
                             \marginpar{\vspace{0.7\baselineskip}
                                        \hspace{-1.1\textwidth}
                                        \fbox{\rm#1}}
                          \fi
                       \fi
                       \label{#1}
                      }
\labelon
 \newcommand{\BbC}{{\mathbb C}}
 \newcommand{\BbZ}{{\mathbb Z}}
 \newcommand{\pdo}{\Psi{\rm DO}}
 \newcommand{\calg}{{\mathfrak g}}
 
 \newcommand{\e}{\varepsilon}

 \newcommand{\eee}{e^{i(\theta-\theta')\cdot\xi}}
 
 \newcommand{\dg}{\dot\gamma}
 \newcommand{\ch}[3]{\Gamma_{{#1}{#2}}^{{#3}}}
 \newcommand{\cc}[3]{c_{{#1}{#2}}^{{#3}}}

 \newcommand{\ints}{\int_{S^1}}
 
 \newcommand{\dtau}{\frac{\partial}{\partial\tau}\left|_{_{_{_{_{_{\tau =
 	      0}}}}}} }
 \newcommand{\ptau}[1]{\partial_\tau^{#1}}
 \newcommand{\xii}[1]{ (\xi^2)^{#1} }
 \newcommand{\eff}[1]{e^{i{#1}\theta}}

 \newcommand{\dir}{\partial\kern-.570em /}
 \newcommand{\dire}{\partial\kern-.570em /{}^{\rm eq}}

 \newcommand{\pa}{\partial}
 \newcommand{\cch}[2]{\Gamma_{{#1}}^{{#2}}}

 \newcommand{\xly}{(X\leftrightarrow Y)}

 \newcommand{\chw}[3]{\Gamma_{{#1}{#2}}^{{#3}} }
 \newcommand{\wgti}{(1+\Delta)^{-1}}
 \newcommand{\wgt}{1+\Delta}
 \newcommand{\ndg}{\nabla_{\dot\gamma}}
 \newcommand{\wgtsi}{(1+\Delta)^{-s}}
 \newcommand{\wgts}{(1+\Delta)^s}
 \newcommand{\ips}[2]{\langle {#1},{#2}\rangle_{s}}
 \newcommand{\ipo}[2]{\langle {#1},{#2}\rangle_{0}}

\begin{document}

%\newpage\null\vskip-4em
%  \noindent\scriptsize{}
%  \scriptsize{
 %%%%%%%%%%%%%%% comment out if undesired and uncomment \textsc{}} line
%  \textsc{AMS Mathematics Subject
%  Classification Numbers: 58J40.
%} }
%\vskip 0.5 in

%\normalsize

\title[The Geometry of Loop Spaces I:  $ {\rm H^s}$-Riemannian Metrics]
{The Geometry of Loop Spaces I:  ${\bf H^s}$-Riemannian Metrics}
\author[Y. Maeda]{Yoshiaki Maeda}
\address{Department of Mathematics\\
Keio University}
\email{maeda@math.keio.ac.jp}
\author[S. Rosenberg]{Steven Rosenberg}
\address{Department of Mathematics and Statistics\\
  Boston University}
\email{sr@math.bu.edu}
\author[F. Torres-Ardila]{Fabi\'an Torres-Ardila}
\address{Center of Science and Mathematics in Context\\
 University of Massachusetts Boston}
\email{fabian.torres-ardila@umb.edu}

\begin{abstract}   A Riemannian metric on a manifold $M$ induces a family of
  Riemannian metrics on the loop space $LM$ depending on a Sobolev space
  parameter $s$.   We compute the connection forms  of these metrics and the higher symbols of their curvature forms, which
  take values in
  pseudodifferential operators ($\Psi$DOs).  These calculations are used in the followup paper \cite{MRT2}
  to construct Chern-Simons classes on $TLM$ which 
detect nontrivial elements in the diffeomorphism group of certain Sasakian $5$-manifolds associated
 to K\"ahler surfaces. 

  \end{abstract}

\maketitle

 \centerline{Dedicated to the memory of Prof. Shoshichi Kobayashi}

\bigskip\bigskip

\bigskip
\section{{\bf Introduction}}

The loop space $LM$\let\thefootnote\relax\footnote{MSC number: 58J40.  Keywords:  loop spaces, Levi-Civita connections, pseudodifferential operators.}
 of a manifold $M$ appears frequently in mathematics and
 mathematical physics. 
 %In particular, path integral proofs  of the Atiyah-Singer index theorem 
% are formulated on loop spaces \cite{A, B, LMRT}. 
In this paper, we develop
 the Riemannian geometry of loop spaces. In a companion paper \cite{MRT2}, we describe a computable theory of characteristic classes for the tangent bundle $TLM$. 
 
Several new features appear for Riemannian geometry on the infinite dimensional manifold $LM$. 
First, for a fixed Riemannian metric $g$ on $M$, there is a natural one-parameter family of metrics
$g^s$  on $LM$ associated to a Sobolev parameter $s >0$.
Our main goal is to compute
 the Levi-Civita connection for $g^s$ and the associated connection and curvature forms.  
 For $s=0$, this is the usual $L^2$ metric, whose connection one-form and curvature two-form are essentially the same as the corresponding forms for $g$.  Second,
for $s >0$ these forms take values in zeroth order pseudodifferential operators ($\pdo$s) acting on sections of a trivial bundle over the circle.  
 Thus the analysis of these $\pdo$s and their symbols is essential to understanding the geometry of $LM$.  In contrast to the finite dimensional case, where a Riemannian metric on $M$ implements a {\it reduction} of the structure group of $TM$, $g^s$ is not compatible with the structure group of $TLM$. This forces an {\it extension} of the structure group of $TLM$ from a gauge group to a group of invertible zeroth order $\pdo$s.
 
 %For loop groups $\Omega G$, the Levi-Civita connections for an invariant metric have been computed in 
% \cite{Freed}, \cite{andres}.  

The paper is organized as follows.   \S2 discusses connections associated to $g^s.$  After some preliminary material on $\pdo$s,
we compute the Levi-Civita connection for $s=0$ (Lemma \ref{lem:l2lc}), $s=1$
(Theorem \ref{old1.6}),  $s\in \Z^+$ (Theorem \ref{thm:sinz}), and general 
$s>\frac{1}{2}$ (Theorem \ref{thm25}).  The extension of the structure group is discussed in \S2.6.
 In \S3, we show that our results extend work of Freed and Larrain-Hubach on loop groups \cite{Freed, andres}.
In the Appendix,
we compute the higher order symbols of the connection and curvature  forms 
of these connections.

One main motivation for this paper is the construction of characteristic and secondary classes on $TLM$, using the Wodzicki residue of $\pdo$s.  In the companion paper \cite{MRT2}, 
we use the main theorems in this paper to construct Chern-Simons classes on $TLM$ which detect that
$\pi_1({\rm Diff}(M^5))$ is infinite, where ${\rm Diff}(M^5)$ is the diffeomorphism group of 
infinite families of Sasakian $5$-manifolds associated to integral K\"ahler surfaces. 
Thus this work relates to a major theme in Prof.~Kobayashi's many important papers, namely the relationship between Riemannian and complex geometry.

Our many discussions with Sylvie Paycha are gratefully
acknowledged. We also thank  Dan Freed for pointing out an error in  an earlier  version of 
the paper.

\section{{\bf The Levi-Civita Connection for Sobolev Parameter $s\geq 0$}}
\label{LCconnection}

In this section, we compute the Levi-Civita connection on $LM$
associated to a Riemannian metric on $M$ and a Sobolev parameter $s=0$ 
or $s>\frac{1}{2}.$  The standard $L^2$ metric on $LM$ is the case $s=0$, and otherwise we avoid technical issues by assuming that $s$ is greater than the critical exponent $\frac{1}{2}$ for analysis on bundles over $S^1.$  
  The
main results are  Lemma \ref{lem:l2lc}, Theorem \ref{old1.6}, 
Theorem \ref{thm:sinz}, and Theorem \ref{thm25},
which compute the Levi-Civita connection for $s =0$, $s=1$, 
$s\in \Z^+$,  and general $s >\frac{1}{2},$ respectively.

%In \S2.1, we review
%material on $LM$, and in \S2.2 we review pseudodifferential operators and the
%Wodzicki residue.  In \S2.3, we give the crucial computations of the Levi-Civita connections
%for $s=0,1$.
%This computation is extended to $s\in \Z^+$ in \S2.4, and to general $s>\frac{1}{2}$ in
%\S2.5.    In \S2.6, we discuss how the geometry of $LM$ forces an extension of
%the structure group of $LM$ from a gauge group to a group of bounded
%invertible $\pdo$s.
     
\subsection{{\bf Preliminaries on $LM$}}

${}$
\medskip

Let $(M, \langle\ ,\ \rangle)$ 
be a closed, connected,  oriented Riemannian $n$-manifold with  loop space $LM
= C^\infty(S^1,M)$ of smooth loops. 
$LM$ is a smooth infinite dimensional Fr\'echet manifold, but it is
 technically simpler 
to work 
with the smooth Hilbert manifold $H^{s'}(S^1,M)$ of loops in some Sobolev class $s' \gg 0,$
as we now recall. For $\gamma\in LM$, the formal
tangent space $T_\gamma LM$ is 
$\Gamma(\gamma^*TM)$, the space
 of smooth sections of the pullback bundle $\gamma^*TM\to
S^1$.   The actual tangent space of $H^{s'}(S^1, M)$ at $\gamma$ is 
$H^{s'-1}(\gamma^*TM),$     the sections of $\gamma^*TM$ of Sobolev class $s'-1.$
We will fix $s'$ and use $LM, T_\gamma LM$ for $H^{s'}(S^1, M), H^{s'-1}(\gamma^*TM)$, respectively.

For each $s>1/2,$ we can complete $\Gamma(\gamma^*TM\otimes \C)$
 with respect to the Sobolev inner product
 \begin{equation}\label{eq:Sob1}
\langle X,Y\rangle_{s}=\frac{1}{2\pi}\int_0^{2\pi} \langle(1+\Delta)^{s}
X(\theta),Y(\theta)
\rangle_{\gamma (\theta)}d\theta,\  X,Y\in \Gamma(\gamma^*TM).
\end{equation}
Here $\Delta=D^*D$, with $D=D/d\gamma$ the covariant derivative along
$\gamma$. (We use this notation instead of the classical $D/d\theta$ to keep track
of $\gamma$.)
We need the complexified pullback bundle $\gamma^*TM\otimes \BbC$, denoted from now on
just as $\gamma^*TM$, in order to apply the
pseudodifferential operator  $(1+\Delta)^{s}.$
The construction of $(1+\Delta)^{s}$ is reviewed in
\S\ref{pdoreview}.  
By the basic elliptic estimate, the  completion of  $\gamma^*TM$ with respect to 
(\ref{eq:Sob1}) is  $H^{s}(\gamma^*TM)$.   We can consider the
$s$ metric on $TLM$ for any $s\in \R$, but we will only consider 
$s=0$ or $1/2 < s\leq s'-1.$

A small real neighborhood $U_\gamma$ 
of the zero section in $H^{s'-1}(\gamma^*TM)$ is a
coordinate chart near $\gamma\in LM$ 
via the pointwise exponential map
\begin{equation}\label{pointwiseexp}
\exp_\gamma:U_\gamma
\to L M, \ X \mapsto 
\left(\theta\mapsto \exp_{\gamma(\theta)} X(\theta)\right).  
\end{equation}
%Note that the domain of the exponential map is not contained in $T_\gamma LM.$
The differentiability of the transition functions $\exp_{\gamma_1}^{-1}\cdot
\exp_{\gamma_2}$ is proved in
\cite{E} and \cite[Appendix A]{Freed1}.
Here $\gamma_1, \gamma_2$ are close loops in the sense that
a geodesically convex neighborhood of $\gamma_1(\theta)$ contains
$\gamma_2(\theta)$ and vice versa for all $\theta.$
Since 
$\gamma^*TM$
is (noncanonically) isomorphic to the trivial bundle ${\mathcal R} =
S^1\times \BbC^n\to S^1$, 
the model space for $LM$ is the set of 
$H^{s'}$ sections of this trivial bundle.  The $s$ metric is a weak Riemannian metric for $s<s'-1$ in the sense that the topology induced on $H^{s'}(S^1, M)$ by the exponential map applied to $H^{s}(\gamma^*TM)$ is weaker than the $H^{s'}$ topology.

The complexified tangent bundle
$TLM$ has
transition functions $d(\exp_{\gamma_1}^{-1} 
\circ \exp_{\gamma_2})$.  Under
the isomorphisms $\gamma_1^*TM \simeq {\mathcal R} \simeq
\gamma_2^*TM$, the transition functions lie in the gauge group
${\calG}({\mathcal R})$, so this is the structure group of $TLM.$

%Keeping track of $s$ and $s'$ is annoying, so from now on we abuse notation and just use
%$s$ for both parameters.

\subsection{{\bf Review of $\pdo$ Calculus}}\label{pdoreview}
${}$
\medskip

We recall the construction of classical
pseudodifferential operators ($\pdo$s)
 on a closed 
manifold $M$ from \cite{gilkey, See}, assuming knowledge of $\pdo$s on
 $\R^n$ (see e.g. \cite{hor, shu}).   
 
 %We emphasize how to
%calculate global symbols in local coordinates, since subprincipal terms
%are coordinate dependent (e.g. (\ref{last})).  

 A linear operator $P:C^\infty(M)\to C^\infty(M)$
is a $\pdo$ of  order $d$ if for
every open chart $U\subset M$ and functions $\phi,\psi\in C_c^\infty(U)$,
$\phi P\psi$ is a $\pdo$ of order $d$ on $\R^n$, where we do not
 distinguish between $U$ and its diffeomorphic image in $\R^n$.
 Let
  $\{U_i\}$ be a finite cover of $M$
with subordinate partition of unity
  $\{\phi_i\}.$  Let $\psi_i\in C^\infty_c(U_i)$ have $\psi_i \equiv 1$ on
  supp$(\phi_i)$  and set $P_i = \psi_iP\phi_i.$
Then
  $ \sum_i \phi_iP_i\psi_i$ is a $\pdo$ on $M$, and
$P$ differs from $ \sum_i \phi_iP_i\psi_i$ by a smoothing operator, denoted
$P\sim \sum_i \phi_iP_i\psi_i$. 
 In particular, this sum is independent of the
  choices up to smoothing operators.
All this carries over to $\pdo$s acting on sections of a bundle over $M$.

An example is the $\pdo$ \ $(1+\Delta-\lambda)^{-1}$ for $\Delta$ a positive
order nonnegative elliptic $\pdo$ and $\lambda$
outside the spectrum of $1+\Delta.$  In each $U_i$, we construct a parametrix
$P_i$ for $A_i = \psi_i
(1+\Delta-\lambda)\phi_i$ by formally inverting $\sigma(A_i)$ and then
constructing a $\pdo$ with the inverted symbol.  By \cite[App.~A]{A-B2},
$ B = \sum_i \phi_iP_i\psi_i$ is a parametrix for $(1+\Delta-\lambda)^{-1}$.
  Since
  $B \sim (1+\Delta-\lambda)^{-1}$,
  $(1+\Delta-\lambda)^{-1}$ is itself a $\pdo$.
For $x\in U_i$, by definition
$$\sigma((1+\Delta-\lambda)^{-1})(x,\xi) = \sigma(P)(x,\xi) = \sigma(\phi
P\phi)(x,\xi),$$
where $\phi$ is a bump function with $\phi(x) = 1$ \cite[p.~29]{gilkey};  
the symbol depends on the choice of $(U_i, \phi_i).$  

The operator $\weight$ for Re$(s) <0$,
which exists as a bounded
operator on $L^2(M)$ by the functional calculus, is also a $\pdo$.
To see this,
 we construct the putative symbol $\sigma_i$
of $ \psi_i\weight\phi_i$ in each $U_i$
by a contour integral $\int_\Gamma \lambda^{s}\sigma[(1+\Delta-\lambda)^{-1}]
d\lambda$
around the spectrum of $1+\Delta$.
We then
construct a $\pdo$ $Q_i$ on $U_i$ with $\sigma(Q_i) = \sigma_i$,
and set $Q  = \sum_i\phi_i Q_i\psi_i.$
By arguments in \cite{See},
$\weight \sim Q$,
so $\weight$ is a $\pdo$.

\subsection{The Levi-Civita Connection for $s=0, 1$}
${}$
\medskip

The smooth Riemannian manifold $LM = H^{s'}(S^1,M)$ has tangent bundle $TLM$ with 
$T_\gamma LM = H^{s'-1}(\gamma^*TM).$  For the $s'-1$ metric on $TLM$ (i.e., 
$s = s'-1$ in (\ref{eq:Sob1})), 
the 
Levi-Civita connection exists and is determined by the six term formula
\begin{eqnarray}\label{5one}
2\ip{\con{s}{X}Y,Z}_{s} &=& X\ip{Y,Z}_{s}+Y\ip{X,Z}_{s}-Z\ip{X,Y}_{s}\\
&&\qquad +\ip{[X,Y],Z}_{s}+\ip{[Z,X],Y}_{s}-\ip{[Y,Z],X}_s\nonumber
\end{eqnarray}
\cite[Ch. VIII]{lang}.  The point is that each term on the RHS of (\ref{5one}) 
is
a {\it continuous} linear functional $T_i:H^{s=s'-1}(\gamma^*TM) \to \BbC$ in $Z$.  Thus 
$T_i(Z) = \ip{T_i'(X,Y),Z}_s$ for a unique $T'(X,Y)\in H^{s'-1}(\gamma^*TM)$, and $\con{s}{Y}X
= \frac{1}{2}\sum_i T'_i.$  

In general, the Sobolev parameter $s$ in (\ref{eq:Sob1}) differs from the parameter $s'$ defining the loop space.  We discuss how this affects the existence of a Levi-Civita connection. 

\begin{rem}\label{lcrem}  For general $s >\frac{1}{2}$, the Levi-Civita connection for the $H^s$ 
metric is guaranteed to exist on the bundle $H^s(\gamma^*TM)$, as above.  However, it is inconvenient to have the bundle depend on the Sobolev parameter, for several reasons:  
(i) $H^s(\gamma^*TM)$ is strictly speaking not the tangent bundle of $LM$, (ii) for the
$L^2$ ($s=0$) metric, the Levi-Civita connection should be given by the Levi-Civita connection on $M$ applied pointwise along the loop (see Lemma \ref{lem:l2lc}), and on $L^2(\gamma^*TM)$  this would have to be interpreted in the distributional sense; (iii) to compute Chern-Simons classes on
$LM$ in \cite{MRT2}, we need to compute with a pair of connections corresponding to $s=0, s=1$ on the
same bundle.  These problems are not fatal: (i) and (ii) are essentially  aesthetic issues,
and for (iii), the connection one-forms will take values in zeroth order $\pdo$s, which are bounded operators on any 
$H^{s'-1}(\gamma^*TM)$, so $s' \gg 0$ can be fixed.  

Thus it is more convenient
to fix $s'$ and consider the family of $H^s$ metrics on $TLM$ for 
$\frac{1}{2} < s < s'-1$. 
  However, the existence of the Levi-Civita connection for the $H^s$ metric is trickier.
  For a sequence $Z\in H^{s'-1} = H^{s'-1}(\gamma^*TM)$ with $Z\to 0$
  in $H^{s'-1}$ or in 
  $H^s$, the RHS of (\ref{5one}) goes to $0$ for fixed $X, Y\in H^s.$  Since
  $H^{s'-1}$ is dense in $H^{s}$, the RHS of (\ref{5one}) extends to a continuous linear functional on $H^s$.  Thus the RHS of (\ref{5one}) is given by
  $\langle L(X,Y), Z\rangle_s$ for some $L(X,Y)\in H^s.$  We set $\nabla^{s}_YX = 
  \frac{1}{2}L(X,Y)$.  Note that even if we naturally demand that
  $X, Y\in H^{s'-1}$, we only get $\nabla^s_YX\in H^s\supset H^{s'-1}$ without additional work.  Part of the content of Theorem \ref{thm25} is that the Levi-Civita connection exists in the {\it strong sense}:  given a tangent vector $X\in H^{s'-1}(\gamma^*TM)$ and a smooth vector field
$Y_\eta\in H^{s'-1}(\eta^*TM)$ for all $\eta$,
   $\nabla^s_XY(\gamma)\in H^{s'-1}(\gamma^*TM).$  See Remark 2.6.
  
  %%  {\it In this technical sense, if $LM$ is modeled on $s'$ maps, the Levi-Civita connection for the $H^s$ metric on $TLM$ may not exist.}  This reflects the fact that
%%  the $H^s$ metric is a weak Riemannian metric for $LM$ if $s < s'-1.$  Nevertheless, we will refer to $\nabla^s_YX$ as the $H^s$ Levi-Civita connection.  

%The existence of the Levi-Civita connection for the $H^s$ metric then follows.  For if 
%$Z\to 0 in $H^{s'-1}$, then $Z\to 0$ in $H^s$.  For fixed $X, Y\in H^s$, the RHS of (\ref{5one}) goes to zero.  Thus the RHS defines a linear functional $\lambda$ of $Z\in H^{s'-1}$, so
%there exists $L(X,Y)\in H^{s'-1} with $\lambda(Z) = \langle L(X,Y), Z\rangle_{s'-1}
%WRONG INNER PRODUCT!
\end{rem}

We need to discuss local coordinates on $LM$.
For motivation, recall that
\begin{equation}\label{lie}[X,Y]^a = X(Y^a)\partial_a - Y(X^a)\partial_a
\equiv \delta_X(Y) -\delta_Y(X)
\end{equation}
in local coordinates on a finite dimensional manifold.  Note that
$X^i\partial_iY^a = X(Y^a) =
(\delta_XY)^a$ in this notation.

Let $Y$ be a vector field on $LM$, and let $X$ be a tangent vector at
$\gamma\in LM.$  The local variation $\delta_XY$ of $Y$ in the direction of $X$ at $\gamma$ is 
defined as usual: let $\gamma(\e,\theta)$ be a family of loops in $M$
with $\gamma(0,\theta) = \gamma(\theta), \frac{d}{d\e}|_{_{\e=0}}
\gamma(\e,\theta) = X(\theta).$  Fix $\theta$, and let $(x^a)$ be
coordinates near $\gamma(\theta)$.  We call these coordinates 
{\it manifold coordinates.} Then
$$\delta_XY^a(\gamma)(\theta) \stackrel{{\rm def}}{=}
\frac{d}{d\e}\biggl|_{_{_{\e =0}}} Y^a(\gamma(\e,\theta)).$$
Note that $\delta_XY^a = (\delta_XY)^a$ by definition.

\begin{rem} Having $(x^a)$ defined only near a fixed $\gamma(\theta)$ is inconvenient.
We can find coordinates that work for all points of $\gamma(\theta)$ as follows. For
  fixed $\gamma$, there is an $\e$ such that for all $\theta$,
  $\exp_{\gamma(\theta)} X$ is inside the cut locus of $\gamma(\theta)$ if
  $X\in T_{\gamma(\theta)}M$ has $|X|<\e.$  Fix such an $\e.$ Call 
  $X\in H^{s'-1}(\gamma^*TM)$ {\it
  short} if $|X(\theta)|<\e$ for all $\theta.$  Then
$$U_\gamma = \{\theta \mapsto \exp_{\gamma(\theta)}X(\theta) | X\ {\rm is\
    short}\}\subset LM$$
is a coordinate neighborhood of $\gamma$ parametrized by $\{ X: X\  {\rm is\ 
    short}\}.$  

 We know
$H^{s'-1}(\gamma^*TM) \simeq H^{s'-1}(S^1\times \R^n)$ noncanonically, so
$U_\gamma$ is parametized by short sections of $H^{s'-1}(S^1\times \R^n)$ for
a different $\e.$  In particular, we have a smooth diffeomorphism $\beta$ from
$U_\gamma$ to short sections of $H^{s'-1}(S^1\times \R^n)$.

Put coordinates $(x^a)$ on $\R^n$, which we identify canonically with the fiber $\R^n_\theta$
over $\theta$ in $S^1\times \R^n$. For $\eta\in
U_\gamma$, we have $\beta(\eta) = (\beta(\eta)^1(\theta),...,\beta(\eta)^n(\theta)).$
As with finite dimensional coordinate systems, we will drop $\beta$ and just
write
$\eta = (\eta(\theta)^a).$ These coordinates work for all
$\eta$ near $\gamma$ and for all $\theta.$  The definition of $\delta_XY$ above carries over to exponential coordinates.

We will call these coordinates {\it exponential coordinates}.
\end{rem}

(\ref{lie}) continues to hold
 for vector fields on $LM$, in either 
 manifold or exponential coordinates.
   To see this, one checks that the coordinate-free proof that $L_XY(f) =
 [X,Y](f)$ for $f\in C^\infty(M)$ (e.g.~\cite[p.~70]{warner}) carries over to
 functions on $LM$.  In brief, the usual proof involves a map $H(s,t)$ of a
 neighborhood of the origin in $\R^2$ into $M$, where $s,t$ are parameters for
 the flows of $X, Y,$ resp.  For $LM$, we have a map $H(s,t,\theta)$, where
 $\theta$ is the loop parameter.  
The  usual proof uses
 only $s, t$ differentiations,
 so $\theta$ is unaffected.    The point is that the $Y^i$ are local functions
 on the $(s,t,\theta)$  parameter space, whereas
the $Y^i$ are not
  local functions on $M$ at points where loops cross or self-intersect.

We first compute the $L^2$ ($s=0$) Levi-Civita connection invariantly and in 
manifold coordinates.

\begin{lem} \label{lem:l2lc} Let $\nabla^{LC}$ be the Levi-Civita connection on $M$.  
 Let $\ev_\theta:LM\to M$ be $\ev_\theta(\gamma) = \gamma(\theta).$ 
 Then $D_XY(\gamma)(\theta) \stackrel{\rm def}{=} 
 (\ev_\theta^*\nabla^{LC})_XY(\gamma)(\theta)$ is the $L^2$ Levi-Civita connection on $LM$.  In manifold coordinates,
 \begin{equation}\label{l2lc} (D_XY)^a(\gamma)(\theta) = \delta_XY^a(\gamma)(\theta) +
  \cch{bc}{a}(\gamma(\theta))X^b(\gamma)(\theta) Y^c(\gamma)(\theta).
  \end{equation}
\end{lem}
\medskip

As in Remark \ref{lcrem}, we may assume that
$X, Y\in H^{s'-1}(\gamma^*TM)$ with $s' \gg 0$, so (\ref{l2lc}) makes sense.

\begin{proof} $\ev_\theta^*\nabla^{LC}$ is a connection on
$\ev_\theta^*TM\to LM$. We have
$\ev_{\theta,*}(X) = X(\theta)$.  If $U$ is a coordinate
neighborhood on $M$ near some $\gamma(\theta)$, then on $\ev_\theta^{-1}(U)$, 
\begin{eqnarray*}(\ev_\theta^*\nabla^{LC})_XY^a(\gamma)(\theta) &=& (\delta_{X}Y)^a(\gamma)(\theta) +
((\ev_\theta^*\omega^{LC}_{X})Y)^a (\theta)\\
&=& (\delta_{X}Y)^a(\gamma)(\theta) + 
  \chw{b}{c}{a}(\gamma(\theta))X^b(\gamma)(\theta) Y^c(\gamma)(\theta).
%&=& (D_XY)^a(\gamma)(\theta).
\end{eqnarray*}
Since $\ev_\theta^*\nabla^{LC}$ is a connection, for each fixed $\theta$, $\gamma$ and $X\in
 T_\gamma LM$, 
 $Y\mapsto$\\
 $ (\ev^*_\theta\nabla^{LC})_XY(\gamma)$
 has Leibniz rule with respect to
functions on $LM$.   Thus $D$ is a connection on $LM.$

$D$ is torsion free, as from the local expression
  $D_XY - D_YX = \delta_XY - \delta_YX = [X,Y].$

To show that  $D_XY$ is compatible with the $L^2$
  metric, first recall that  for a function $f$ on $LM$, $D_Xf = \delta_Xf =
  \frac{d}{d\e}|_{_{\e=0}}f(\gamma(\e,\theta))$ for $X(\theta)
   = \frac{d}{d\e}|_{_{\e=0}}\gamma(\e, \theta).$
  (Here $f$ depends only on
  $\gamma$.)  Thus (suppressing the partition of unity, which is independent of $\e$)
\begin{eqnarray*} D_X\langle Y,Z\rangle_0 &=& 
  \frac{d}{d\e}\biggl|_{_{_{\e=0}}}\int_{S^1} g_{ab}(\gamma(\e,\theta))
Y^a(\gamma(\e,\theta))Z^b(\gamma(\e,\theta))d\theta\\
&=& \int_{S^1}\partial_c
g_{ab}(\gamma(\e,\theta))
X^cY^a(\gamma(\e,\theta))Z^b(\gamma(\e,\theta))d\theta\\
&&\qquad + \int_{S^1} 
g_{ab}(\gamma(\e,\theta))
(\delta_XY)^a(\gamma(\e,\theta))Z^b(\gamma(\e,\theta))d\theta\\
&&\qquad 
+ \int_{S^1} 
g_{ab}(\gamma(\e,\theta))
Y^a(\gamma(\e,\theta))(\delta_XZ)^b(\gamma(\e,\theta))d\theta\\
&=& \int_{S^1}\Gamma{}_{c a}^{e}g_{eb} X^cY^aZ^b +
\Gamma{}_{c b}^{e}g_{ae}X^cY^aZ^b\\
&&\qquad +g_{ab}(\delta_XY)^aZ^b + g_{ab}Y^a(\delta_X Z)^bd\theta\\
&=& \langle D_XY,Z\rangle_0 + \langle Y, D_XZ\rangle_0.
\end{eqnarray*}
\end{proof}

\begin{rem} The local expression for $D_XY$ also holds in exponential coordinates. More precisely, let $(e_1(\theta),...,e_n(\theta))$
be a global frame of $\gamma^*TM$ given by the trivialization of
$\gamma^*TM.$  Then $(e_i(\theta))$ is also naturally a frame of
$T_XT_{\gamma(\theta)}M$ for all $X\in T_{\gamma(\theta)}M.$  We use
$\exp_{\gamma(\theta)}$ to pull back the metric on $M$ to a metric on
$T_{\gamma(\theta)}M$: 
$$g_{ij}(X) = (\exp^*_{\gamma(\theta)}g)(e_i, e_j) =
  g(d(\exp_{\gamma(\theta)})_X (e_i),   d(\exp_{\gamma(\theta)})_X
  (e_j))_{\exp_{\gamma(\theta)}X}.$$
Then the Christoffel symbols  
$\Gamma_{b c}^{a}(\gamma(\theta))$the term {\it e.g.,} 
are computed with respect to
%WHY IS THIS CAUSING PROBLEMS?%%%%%%%%%%%%%%%%%%
  this metric.  For example, the term $\partial_\ell g_{bc}$ means $e_\ell
  g(e_a, e_b)$, etc.  The proof that $D_XY$ has the local expression (\ref{l2lc}) 
  then carries over to exponential coordinates.

\end{rem}

The  $s=1$ Levi-Civita connection on $LM$ is given as follows.

\begin{thm} \label{old1.6}  
The $s=1$ Levi-Civita connection $ \nabla^1$ on $LM$ is given at the loop
$\gamma$ by
\begin{eqnarray*}  \nabla^1_XY &=& D_XY + \frac{1}{2}\wgti\left[
-\ndg(R(X,\dg)Y) - R(X,\dg)\ndg Y\right.\\
&&\qquad \left. -\ndg(R(Y,\dg)X) - R(Y,\dg)\ndg X\right.\\
&&\qquad \left. +R(X,\ndg Y)\dg - R(\ndg X, Y)\dg\right].
\end{eqnarray*}
\end{thm}

On the right hand side of this formula, the term
 $\ndg(R(X,\dg)Y)$ denotes the vector field along $\gamma$ whose
value at $\theta$ is $\nabla^{LC}_{\dg(\theta)}R(X(\theta),\dg(\theta))Y(\theta).$

We prove this in a series of steps.  The assumption in the next Proposition will be dropped later.

\begin{prop} \label{old1.3}
The Levi-Civita connection for the $s=1$ metric is given by
$$\nabla_X^1Y = D_XY + \frac{1}{2}\wgti[D_X, 1+\Delta]Y +
  \frac{1}{2}\wgti[D_Y, 1+\Delta]X 
+ A_XY,$$
where  we assume that for $X, Y\in H^{s'-1}$, $A_XY$ is well-defined by
\begin{equation}\label{insert2}-\frac{1}{2}\langle [D_Z,\wgt]X,Y\rangle_0 = \langle A_XY,Z\rangle_1.
\end{equation}
\end{prop}

\begin{proof} By Lemma \ref{lem:l2lc},
\begin{eqnarray*} X\langle Y,Z\rangle_1 &=& X\langle (\wgt)Y,Z\rangle_0 =
  \langle D_X((\wgt)Y),Z\rangle_0 + \langle (\wgt)Y, D_XZ\rangle_0\\
Y\langle X,Z\rangle_1 &=& \langle D_Y((\wgt)X),Z\rangle_0 + \langle (\wgt)X,
D_YZ\rangle_0\\
-Z\langle X,Y\rangle_1 &=& -\langle D_Z((\wgt)X),Y\rangle_0 - \langle (\wgt)X,
D_ZY\rangle_0\\
\langle [X,Y],Z\rangle_1 &=& \langle(\wgt)(\delta_XY - \delta_YX), Z\rangle_0
= \langle (\wgt)(D_XY - D_YX),Z\rangle_0\\
\langle[Z,X],Y\rangle_1 &=& \langle(\wgt)(D_ZX-D_XZ),Y\rangle_0\\
-\langle[Y,Z],X\rangle_1 &=& -\langle(\wgt)(D_YZ-D_ZY),X\rangle_0.
\end{eqnarray*}
The six terms on the left hand side must sum up to $2\langle \nabla^1_XY,Z\rangle_1$ 
in the sense of Remark \ref{lcrem}.
After some cancellations,  we get
\begin{eqnarray*} 2\langle\nabla^1_XY,Z\rangle_1 &=& \langle D_X((\wgt)Y),Z\rangle_0 +
  \langle D_Y((\wgt)X),Z\rangle_0\nonumber\\
&&\qquad + \langle (\wgt)(D_XY - D_YX),Z\rangle_0 - \langle
  D_Z((\wgt)X),Y\rangle_0\nonumber\\
&&\qquad +\langle(\wgt)D_ZX),Y\rangle_0\nonumber\\
&=& \langle (\wgt)D_XY,Z\rangle_0 + \langle [D_X,\wgt] Y, Z\rangle_0\\
&&\qquad + \langle (\wgt)D_YX,Z\rangle_0 + \langle [D_Y,\wgt] X, Z\rangle_0\nonumber\\
&&\qquad + \langle (\wgt)(D_XY - D_YX),Z\rangle_0 -\langle
  [D_Z,\wgt]X,Y\rangle_0\\
%  \end{eqnarray*}
%Thus
%\begin{eqnarray}\label{insert} 2\langle\nabla_XY,Z\rangle_1 
&=& 2\langle D_XY,Z\rangle_1 + \langle \wgti[D_X,\wgt]Y,Z\rangle_1\\
&&\qquad + \langle \wgti[D_Y,\wgt]X,Z\rangle_1 +2
%\langle   [D_Z,\wgt]X,Y\rangle_0.\nonumber
\langle A_XY,Z\rangle_1.
\end{eqnarray*}

%The right hand side of  (\ref{insert}) is a continuous linear functional
%of $Z\in H^1(\gamma^*TM)$, as in (\ref{5one}).  Since each term
%^except the last is obviously a continuous linear functional of $Z$, 
%so is the last term.  Therefore there exists $A_XY\in H^1(\gamma^*TM)$ satisfying (\ref{insert2}).
\end{proof}

Now we compute the bracket terms in the Proposition.  We have $[D_X,\wgt] =
[D_X,\Delta]$. For $\dg = (d/d\theta)\gamma$, 
$$0 = \dot\gamma\langle X, Y\rangle_0 = \langle\nabla_{\dot\gamma}X,Y\rangle_0
+ \langle X,\nabla_{\dot\gamma}Y\rangle_0,$$
so 
\begin{equation}\label{one}\Delta = \nabla_{\dot\gamma}^* \nabla_{\dot\gamma}
  = -\nabla_{\dot\gamma}^2.
\end{equation}

\begin{lem} $[D_X,\nabla_{\dot\gamma}]Y = R(X,\dot\gamma)Y.$
\end{lem}

\begin{proof} 
Note that $\gamma^\nu, \dot\gamma^\nu$ are locally defined functions on 
$S^1\times LM.$
Let $\tilde\gamma:
[0,2\pi]\times (-\e,\e)\to M$ be a smooth map with $\tilde\gamma(\theta,0) =
\gamma(\theta)$, and
$\frac{d}{d\tau}|_{\tau = 0}\tilde\gamma(\theta,\tau) = Z(\theta).$
Since $(\theta,\tau)$ are coordinate functions on
$S^1\times (-\e,\e)$, we have
\begin{eqnarray}\label{badterms} Z(\dg^\nu) &=& \delta_Z(\dg^\nu) = \ptau{Z}(\dg^\nu) =
\dtau\right.\left(\frac{\partial}{\partial\theta}
(\tilde\gamma(\theta,\tau)^\nu\right)\\
&=& \frac{\partial}{\partial\theta}
\dtau\right. \tilde\gamma(\theta,\tau)^\nu = \partial_\theta Z^\nu \equiv
\dot Z^\nu.\nonumber
\end{eqnarray}

%Recall that $\delta_X\dg^j = \dot X^j.$
 We compute
\begin{eqnarray*} 
%\lefteqn{
(D_X\nabla_{\dg} Y)^a %}\\
&=& \delta_X(\nabla_{\dg} Y)^a +
  \chw{b}{c}{a}X^b\nabla_{\dg} Y^c\\
&=& \delta_X(\dg^j\partial_jY^a + \chw{b}{c}{a}\dg^bY^c)%\\
%&&\qquad 
+ \chw{b}{c}{a}X^b(\dg^j\partial_jY^c + \chw{e}{f}{c}\dg^e Y^f)\\
&=& \dot X^j\partial_jY^a + \dg^j\partial_j\delta_XY^a +
  \partial_m\chw{b}{c}{a}X^m\dg^bY^c%\\
%&&\qquad 
+ \chw{b}{c}{a}\dot X^bY^c + \chw{b}{c}{a}\dg^b\delta_XY^c\\
&&\qquad 
+ \chw{b}{c}{a}X^b\dg^j\partial_jY^c +
  \chw{b}{c}{a}\chw{e}{f}{c}X^b\dg^eY^f.\\
(\nabla_{\dg} D_XY)^a &=& \dg^j(\partial_j(D_XY)^a +
  \chw{b}{c}{a}\dg^b (D_XY)^c)\\
&=& \dg^j\partial_j(\delta_XY^a + \chw{b}{c}{a}X^b Y^c) 
+ \chw{b}{c}{a}\dg^b(\delta_XY^c + \chw{s}{f}{c}X^eY^f)\\
&=& \dg^j\partial_j\delta_XY^a + \dg^j\partial_j\chw{b}{c}{a}X^bY^c +
  \chw{b}{c}{a}\dot X^bY^c
+ \chw{b}{c}{a}X^b\dot Y^c + \chw{b}{c}{a}\dg^b\delta_XY^c \\
&&\qquad + \chw{b}{c}{a}\chw{e}{f}{c}\dg^b X^eY^f.
\end{eqnarray*}
Therefore
\begin{eqnarray*} (D_X\nabla_{\dg}Y - \nabla_{\dg}D_XY)^a &=& \partial_m
  \chw{b}{c}{a}X^m\dg^bY^c - \partial_j \chw{b}{c}{a}\dg^j X^bY^c
  + \chw{b}{c}{a}\chw{e}{f}{c}X^b\dg^e Y^f   \\
&&\qquad 
-\chw{b}{c}{a}\chw{e}{f}{c}\dg^b X^e Y^f \\
&=& (\partial_j \Gamma_{bc}^{a} - \partial_b \chw{j}{c}{a} +\chw{j}{e}{a}\chw{b}{c}{e}-
\chw{b}{e}{a}\chw{j}{c}{e})\dg^b X^j Y^c \\
&=& R_{jbc}^{\ \ \ a}X^j\dg^b Y^c,
\end{eqnarray*}
so 
$$D_X\nabla_{\dg}Y - \nabla_{\dg}D_XY =  R(X,\dg)Y.$$
\end{proof}

\begin{cor}\label{cor:zero}
 At the loop $\gamma$, $[D_X,\Delta]Y = -\nabla_{\dg}(R(X,\dot\gamma)Y) -
  R(X,\dg)\nabla_{\dg}Y.$  In particular, $[D_X,\Delta]$ is a zeroth order
  operator.
\end{cor}

\begin{proof}  
\begin{eqnarray*} [D_X,\Delta]Y &=& (-D_X\ndg\ndg + \ndg\ndg D_X)Y \\
&=& -(\ndg D_X\ndg Y+ R(X,\dg)\ndg Y) +\ndg\ndg D_XY\\
&=& -(\ndg\ndg D_XY + \ndg(R(X,\dg)Y) + R(X,\dg)\ndg Y) 
+\ndg\ndg D_XY\\
&=& -\ndg(R(X,\dg)Y) - R(X,\dg)\ndg Y.
\end{eqnarray*}
\end{proof}

Now we complete the proof of Theorem \ref{old1.6}, showing in the process that $A_XY$ exists. 
%In the proof, we justify that $A_XY$ in 
%Proposition \ref{old1.3} exists.
\medskip

\noindent {\it Proof of Theorem \ref{old1.6}.}
 By Proposition \ref{old1.3} and Corollary \ref{cor:zero}, we have
\begin{eqnarray*} \nabla^1_XY &=& D_XY + \frac{1}{2}\wgti[D_X,\wgt]Y +
  (X\leftrightarrow Y) + A_XY\\
&=& D_XY + \frac{1}{2}\wgti(-\ndg(R(X,\dg)Y) - R(X,\dg)\ndg Y) + 
  (X\leftrightarrow Y) + A_XY,
\end{eqnarray*}
where $(X\leftrightarrow Y)$ denotes the previous term with $X$ and $Y$ switched.

The curvature tensor satisfies 
$$-\langle Z, R(X,Y)W\rangle = \langle R(X,Y)Z,W\rangle = \langle R(Z,W)X,Y
\rangle$$
 pointwise, so
\begin{eqnarray*} \langle A_XY,Z\rangle_1 &=&
  -\frac{1}{2}\langle[D_Z,\wgt]X,Y\rangle_0\\
&=& -\frac{1}{2}\langle (-\ndg(R(Z,\dg)X) - R(Z,\dg)\ndg X,Y\rangle_0\\
&=& -\frac{1}{2} \langle R(Z,\dg)X,\ndg Y\rangle_0 + \frac{1}{2}\langle
  R(Z,\dg)\ndg X,Y\rangle_0\\
&=& -\frac{1}{2} \langle R(X,\ndg Y)Z,\dg\rangle_0 + \frac{1}{2}\langle R(\ndg X,
  Y)Z,\dg\rangle_0\\
&=& \frac{1}{2}\langle Z, R(X,\ndg Y)\dg\rangle_0 - \frac {1}{2} \langle Z,
  R(\ndg X,Y)\dg\rangle_0\\
&=&\frac{1}{2}\langle Z, \wgti(R(X,\ndg Y)\dg - R(\ndg X, Y)\dg)\rangle_1.
\end{eqnarray*}
Thus $A_XY$ must equal $\frac{1}{2} \wgti(R(X,\ndg Y)\dg - R(\ndg X, Y)\dg).$  
This makes sense: for $X, Y\in H^{s'-1}$, 
 $A_XY\in H^{s'}\subset H^1,$ since $R$ is zeroth order.  
 %It 
%follows that $A_XY$ takes $H^1$ to $H^1$ continuously as a map on either $X$ or
% $Y$. Thus
%taking the $H^1$ inner product with $A_XY$ is a continuous linear functional on $H^1.$
\hfill$\Box$
\medskip

\begin{rem} Locally on $LM$, we should have 
$D_XY = \delta_X^{LM}Y +  \omega_X^{LM}(Y)$.  
%In particular, in exponential coordinates we should have 
%$(D_XY)^a = (\delta_X^{LM}Y)^a +  (\omega_X^{LM}(Y))^a$.  
Now
$\delta_X^{LM}Y$  can only mean $\frac{d}{d\tau}|_{\tau =
  0}\frac{d}{d\epsilon}|_{\epsilon = 0}\gamma(\epsilon,\tau,\theta)$, where
$\gamma(0,0,\theta) = \gamma(\theta)$, $\frac{d}{d\epsilon}|_{\epsilon =
  0}\gamma(\epsilon,0,\theta) = X(\theta)$, 
$\frac{d}{d\tau}|_{\tau = 0}\gamma(\epsilon,\tau,\theta) = Y_{\gamma(\epsilon, 0,\cdot)}(\theta).$
 In other words, $\delta_X^{LM}Y$ equals $ \delta_XY$.
Since $D_XY^a = \delta_XY^a + \chw{b}{c}{a}(\gamma(\theta))$, the connection one-form 
$\omega^{LM}$ for the $L^2$ Levi-Civita connection on $LM$ is related to the connection one-form
$\omega^M$ for the Levi-Civita connection on $M$ by
$$\omega^{LM}_X(Y)^a(\gamma)(\theta) = \chw{b}{c}{a}(\gamma(\theta))X^bY^c
= \omega^M_X(Y)^a(\gamma(\theta)).$$
\end{rem}

By this remark, we get
\begin{cor}\label{cor2} The connection one-form $\omega^1$ for $\nabla^1$ in 
exponential coordinates is
\begin{eqnarray}\label{two}\omega^1_X(Y)(\gamma)(\theta) &=& \omega^M_X(Y)(\gamma(\theta))
  +  
\frac{1}{2}\bigl\{\wgti\left[
-\ndg(R(X,\dg)Y) - R(X,\dg)\ndg Y\right.\nonumber\\
&&\qquad \left. -\ndg(R(Y,\dg)X) - R(Y,\dg)\ndg X\right.\\
&&\qquad  \left.+R(X,\ndg Y)\dg - R(\ndg X, Y)\dg\right]\bigr\}(\theta).\nonumber
\end{eqnarray}
\end{cor}

From this Corollary, we can directly compute the curvature of the $s=1$ metric connection.  The result is unpleasant. Fortunately, for the characteristic classes discussed in \cite{MRT2}, we only need the higher order symbols of the curvature.  These are computed in Appendix A.

\subsection{The Levi-Civita Connection for $s\in\BbZ^+$}
${}$
\medskip

For $s>\frac{1}{2}$, the proof of Prop.~\ref{old1.3} extends directly to give

\begin{lem} \label{lem: LCs}
The Levi-Civita connection for the $H^s$ metric is given by
$$\nabla_X^sY = D_XY + \frac{1}{2}\wgtsi[D_X, \wgts]Y +
  \frac{1}{2}\wgtsi[D_Y, \wgts]X 
+ A_XY,$$
where we assume that for $X, Y\in H^{s'-1}$, $A_XY\in H^s$ is characterized by
\begin{equation}\label{axy}
-\frac{1}{2}\langle [D_Z,\wgts]X,Y\rangle_0 = \langle A_XY,Z\rangle_s.
\end{equation}
%where for $X, Y\in H^s$, $A_XY\in H^s$ is characterized by
%%$Z\mapsto -\frac{1}{2}\langle [D_Z,\wgts]X,Y\rangle_0$ is a continuous linear functional on
%$H^s.$
\end{lem}
\bigskip

We now compute the bracket terms.

\begin{lem}\label{bracketterms}
For $s\in \Z^+$, at the loop $\gamma$,
\begin{equation}\label{bracket}
[D_X,\wgts]Y = \sum_{k=1}^s(-1)^k\left(\begin{array}{c}s\\k\end{array}\right)
\sum_{j=0}^{2k-1} \nabla_{\dg}^j(R(X,\dg)\nabla_{\dg}^{2k-1-j}Y).
\end{equation}
In particular, $[D_X,\wgts]Y$ is a $\pdo$ of order at most $2s-1$ in either $X$ or $Y$.
\end{lem}

\begin{proof}  The sum over $k$ comes from the binomial expansion of $\wgts$, so
we just need an inductive formula for 
$[D_X,\Delta^s].$   
The case $s=1$ is Proposition \ref{old1.3}. For the induction step, we have
%%%%%%%%%%%%   PUT INY 
\begin{eqnarray*} [D_X,\Delta^s] &=& D_X\Delta^{s-1}\Delta - \Delta^sD_X\\
&=& \Delta^{s-1}D_X\Delta + [D_X,\Delta^{s-1}]\Delta - \Delta^sD_X\\
&=& \Delta^sD_X +\Delta^{s-1}[D_X,\Delta] + [D_X,\Delta^{s-1}]\Delta
-\Delta^sD_X\\
&=& \Delta^{s-1}(-\nabla_{\dg}(R(X,\dg)Y) -R(X,\dg)\nabla_{\dg}Y)\\
&&\qquad  -  
\sum_{j=0}^{2s-3}(-1)^{s-1}
\nabla^j_{\dg}(R(X,\dg)\nabla_{\dg}^{2k-j-1}(-\nabla^2_{\dg}Y)\\
&=& (-1)^{s-1}(-\nabla_{\dg}^{2s-1}(R(X,\dg)Y) - (-1)^{s-1}\nabla_{\dg}^{2s-2}(R(X,\dg)\nabla_{\dg}Y)\\
&&\qquad + \sum_{j=0}^{2s-3}(-1)^{s}
\nabla^j_{\dg}(R(X,\dg)\nabla_{\dg}^{2k-j-1}(-\nabla^2_{\dg}Y)\\
&=& \sum_{j=0}^{2s-1}(-1)^s \nabla_{\dg}^j(R(X,\dg)\nabla_{\dg}^{2k-1-j}Y).
\end{eqnarray*} 
\end{proof}

%%%%%%%%%%
We check that $A_XY$ is a $\pdo$ in $X$ and $Y$ for $s\in \BbZ^+.$

\begin{lem} \label{insert3} For $s\in\BbZ^+$ and fixed $X, Y\in H^{s'-1}$, $A_XY$ in (\ref{axy})
 is an explicit $\pdo$ in $X$ and $Y$ of order at most $-1.$
 %a continuous linear map from $Z\in H^s$ to $H^s$.  Thus $A_XY\in H^s$ is well defined.
\end{lem}

\begin{proof}  By (\ref{bracket}), for  $j, 2k-1-j \in \{0,1,...,2s-1\}$, a typical term on
the left hand side of (\ref{axy}) is
\begin{eqnarray*}   \ipo{\nabla^j_{\dg}(R(Z,\dg)\nabla_{\dg}^{2k-1-j}X)}{Y} &=& 
(-1)^j
 \ipo{R(Z,\dg)\nabla_{\dg}^{2k-1-j}X}{\nabla^j_{\dg} Y}\\
&=& (-1)^j\ints g_{i\ell} (R(Z,\dg)\nabla_{\dg}^{2k-1-j}X)^i(\nabla^j_{\dg} Y)^\ell d\theta\\
&=& (-1)^j\ints g_{i\ell} Z^k R_{krn}^{\ \ \ i}\dg^r (\ndg^{2k-1-j}X)^n (\ndg^jY)^\ell d\theta\\
&=& (-1)^j
 \ints g_{tm}g^{kt}  g_{i\ell} Z^m R_{krn}^{\ \ \ i}\dg^r (\ndg^{2k-1-j}X)^n (\ndg^jY)^\ell d\theta\\
&=& (-1)^j \ipo{Z}
{g^{kt} g_{i\ell} R_{krn}^{\ \ \ i}\dg^r (\ndg^{2k-1-j}X)^n (\ndg^jY)^\ell\pa_t}\\ 
&=& (-1)^j\ipo{Z}{R^t_{\ rn\ell} \dg^r (\ndg^{2k-1-j}X)^n (\ndg^jY)^\ell\pa_t}\\ 
&=&(-1)^{j+1} \ipo{Z}{R^{\ \ \ t}_{n\ell r} \dg^r (\ndg^{2k-1-j}X)^n (\ndg^jY)^\ell\pa_t}\\
&=& (-1)^{j+1}\ipo{Z}{R(\ndg^{2k-1-j}X,\ndg^jY)\dg}\\
&=& (-1)^{j+1} \ips{Z}{\wgtsi R(\ndg^{2k-1-j}X,\ndg^jY)\dg}.
  \end{eqnarray*}
  (In the integrals and inner products, the local expressions are in fact globally defined one-forms on $S^1$, resp.~vector fields along $\gamma$, so we do not need a partition of unity.)
$\wgtsi R(\ndg^{2k-1-j}X,\ndg^jY)\dg$ is of order at most $-1$ in either $X$ or $Y$, so this term is in 
$H^{s'}\subset H^s.$  Thus the last inner product is well defined.  
\end{proof}

By  (\ref{axy}), (\ref{bracket}) and the proof of Lemma \ref{insert3}, we get
$$A_XY = \sum_{k=1}^s(-1)^k\left(\begin{array}{c}s\\k\end{array}\right)
\sum_{j=0}^{2k-1} (-1)^{j+1}   \wgtsi R(\ndg^{2k-1-j}X,\ndg^jY)\dg.$$
This gives:

\begin{thm} \label{thm:sinz}
For $s\in\BbZ^+$, the Levi-Civita connection for the $H^s$ metric at the
loop $\gamma$ is given by
\begin{eqnarray*} \nabla_X^sY(\gamma) &=& D_XY(\gamma) + \frac{1}{2}\wgtsi
 \sum_{k=1}^s(-1)^k\left(\begin{array}{c}s\\k\end{array}\right)
\sum_{j=0}^{2k-1} \nabla_{\dg}^j(R(X,\dg)\nabla_{\dg}^{2k-1-j}Y)\\
&&\qquad + \xly\\
  &&\qquad 
 +  \sum_{k=1}^s(-1)^k\left(\begin{array}{c}s\\k\end{array}\right)
\sum_{j=0}^{2k-1} (-1)^{j+1}   \wgtsi R(\ndg^{2k-1-j}X,\ndg^jY)\dg.
\end{eqnarray*}
\end{thm}

\subsection{The Levi-Civita Connection for General $s>\frac{1}{2}$}

${}$
\medskip

In this subsection, we show that the $H^s$ Levi-Civita connection for general $s>\frac{1}{2}$ exists in the strong sense of Remark \ref{lcrem}.
The formula is less explicit than in the 
$s\in \Z^+$ case, but is good enough for symbol calculations.

By Lemma \ref{lem: LCs}, we have to examine the term $A_XY$, which, if it exists, is
 characterized by (\ref{axy}):
$$-\frac{1}{2}\ipo{[D_Z,\weight]X}{Y} = \ips{A_XY}{Z}$$
for $Z\in H^s$.  As explained in Remark \ref{lcrem}, we may take
$X, Y\in H^{s'-1}.$
Throughout this section we assume that $s'\gg s$.

The following lemma extends Lemma \ref{bracketterms}.
\begin{lem}\label{pdo}
 (i)   For fixed $Z\in H^{s'-1}$,  $[D_Z,\weight] X$ is a $\Psi$DO of
  order $2s-1$ in $X$. For ${\rm Re}(s)\neq 0$, the principal symbol of $[D_Z,\weight]$ is 
  linear in $s$.
  
  (ii) For fixed $X\in H^{s'-1}$, $[D_Z,\weight]X$ is a $\pdo$ 
  of order $2s-1$   in $Z$.
\end{lem}

As usual, ``of order $2s-1$" means ``of order at most $2s-1.$"

\begin{proof}
(i) For $f:LM\to \BbC$, we get $[D_Z,\weight]fX = f[D_Z,\weight]X$, since $[f,\weight]=0.$  
Therefore, $[D_Z,\weight]X$ depends only on $X|_\gamma.$

By Lemma \ref{lem:l2lc}, $D_Z = \delta_Z + \Gamma \cdot Z$ in shorthand exponential 
coordinates.  The Christoffel symbol term is zeroth order and $\weight$ has scalar leading order symbol, so $[\Gamma\cdot Z,\weight]$ has order $2s-1.$  

From the integral expression for 
$\weight$, it is immediate that 
\begin{eqnarray}\label{immediate}
[\delta_Z,\weight]X &=& (\delta_Z\weight) X + \weight\delta_Z X - \weight\delta_ZX\\
&=& (\delta_Z\weight) X.\nonumber
\end{eqnarray}
$\delta_Z\weight$ is a limit of differences of $\pdo$s on bundles isomorphic to $\gamma^*TM$.
Since the algebra of $\pdo$s is closed in the Fr\'echet topology
of all $C^k$ seminorms
of symbols and smoothing terms
on compact sets, $\delta_Z\weight$ is a $\pdo.$

Since $\weight$ has order $2s$ and has scalar leading order symbol, 
% $[\delta_Z,\weight]$ and hence
$[D_Z,\weight]$ have order $2s-1$.  For later purposes (\S3.2), we compute some explicit symbols.  

Assume Re$(s)<0.$  As in the construction of $\weight$,
 we will compute what the symbol asymptotics
of $\delta_Z\weight$ should
be, and then construct an operator with these asymptotics.
From the functional calculus for unbounded operators, we have
\begin{eqnarray}\label{funcalc}
\delta_Z\weight &=& \delta_Z\left(\frac{i}{2\pi}\int_\Gamma
\lambda^s(1+\Delta-\lambda)^{-1}d\lambda\right)\nonumber\\
&=& \frac{i}{2\pi}\int_\Gamma
\lambda^s\delta_Z (1+\Delta-\lambda)^{-1}d\lambda\\
&=& -\frac{i}{2\pi}\int_\Gamma
\lambda^s (1+\Delta-\lambda)^{-1} (\delta_Z\Delta)
(1+\Delta-\lambda)^{-1}d\lambda,\nonumber
\end{eqnarray}
where $\Gamma$ is a contour around the spectrum of $1+\Delta$, and the
hypothesis on $s$ justifies the exchange of $\delta_Z$ and the integral.  The
operator $A =
(1+\Delta-\lambda)^{-1} \delta_Z\Delta  (1+\Delta-\lambda)^{-1}$ is a $\pdo$
of order $-3$
 with top order symbol
\begin{eqnarray*} \sigma_{-3}(A)(\theta,\xi)^\ell_j &=&
(\xi^2-\lambda)^{-1}\delta^\ell_k (-2Z^i\partial_i\ch{\nu}{\mu}{k}\dg^\nu
   -2\ch{\nu}{\mu}{k}\dot Z^\nu) \xi (\xi^2-\lambda)^{-1}\delta^\mu_j\\
&=&
(-2Z^i\partial_i\ch{\nu}{j}{\ell}\dg^\nu
   -2\ch{\nu}{j}{\ell}\dot Z^\nu)
\xi (\xi^2-\lambda)^{-2}.
\end{eqnarray*}
Thus the top order symbol of $\delta_Z\weight$ should be
\begin{eqnarray}\label{tsmo}
 \sigma_{2s-1}(\delta_Z\weight)(\theta,\xi)^\ell_j
&=& -\frac{i}{2\pi}\int_\Gamma
\lambda^s (-2Z^i\partial_i\ch{\nu}{j}{\ell}\dg^\nu
   -2\ch{\nu}{j}{\ell}\dot Z^\nu)
\xi (\xi^2-\lambda)^{-2} d\lambda  \nonumber\\
&=& \frac{i}{2\pi}\int_\Gamma s\lambda^{s-1}
(-2Z^i\partial_i\ch{\nu}{j}{\ell}\dg^\nu
   -2\ch{\nu}{j}{\ell}\dot Z^\nu)
\xi (\xi^2-\lambda)^{-1} d\lambda \nonumber\\
&=& s(-2Z^i\partial_i\ch{\nu}{j}{\ell}\dg^\nu
   -2\ch{\nu}{j}{\ell}\dot Z^\nu)\xi (\xi^2-\lambda)^{s-1}.
\end{eqnarray}
Similarly, all the terms in the symbol asymptotics for $A$ are of the form
$B^\ell_j \xi^n(\xi^2-\lambda)^m$ for some matrices $B^\ell_j =
B^\ell_j(n,m).$   This produces a symbol sequence $
\sum_{k\in \Z^+}\sigma_{2s-k}$,  and there exists a  $\pdo$ $P$ with $\sigma(P) =
\sum \sigma_{2s-k}$.  (As in \S\ref{pdoreview}, we 
first produce operators $P_i$ on a coordinate cover $U_i$ of $S^1$, 
and then set
$P = \sum_i\phi_iP_i\psi_i$.) The construction
depends on
the choice of local coordinates
covering $\gamma$, the partition of unity and cutoff
functions as above, and a cutoff function in $\xi$; as
usual, different choices change the operator by a smoothing operator.
Standard estimates 
show that $P-\delta_Z\weight$ is a smoothing
operator, which verifies explicitly that  $\delta_Z\weight$ is a $\pdo$ of order $2s-1.$

For Re$(s) >0$,  motivated by differentiating $\weightinv\circ\weight = {\rm
  Id}$, we set
\begin{equation}\label{abc}
\delta_Z\weight = -\weight\circ\delta_Z\weightinv\circ\weight.
\end{equation}
This is again a $\pdo$ of order $2s-1$ with principal symbol linear in $s$. 

 (ii) As a $\pdo$ of order $2s$, $\weight$ has the expression
 $$\weight X(\gamma)(\theta) = \int_{T^*S^1}
 e^{i(\theta-\theta')\cdot \xi} %\tilde 
 p(\theta,\xi) X(\gamma)(\theta')d\theta' d\xi,$$
 where we omit the cover of $S^1$ and its partition of unity on the right hand side.
 Here $%\tilde 
 p(\theta,\xi)$ is the symbol of $\weight$, which has the asymptotic expansion
 $$%\tilde 
 p(\theta,\xi) \sim \sum_{k=0}^\infty p_{2s-k }(\theta,\xi).$$
 The covariant derivative along $\gamma$ on
$Y\in\Gamma(\gamma^*TM)$ is given by
\begin{eqnarray*}\frac{DY}{d\gamma} &=&
(\gamma^*\nabla^{M})_{\partial_\theta}(Y) =
\partial_\theta Y + (\gamma^*\omega^{M})(\partial_\theta)(Y)\\
&=& \partial_\theta(Y^i)\partial_i + \dot\gamma^t Y^r
\Gamma^j_{tr}\partial_j,
\end{eqnarray*}
where $\nabla^{M}$ is the Levi-Civita connection on $M$ and $\omega^{M}$ is the
connection one-form in exponential coordinates on $M$. 
%and $\Gamma^j_{tr}$ are the Christoffel
%symbols.  
For $\Delta =
(\frac{D}{d\gamma})^* \frac{D}{d\gamma}$, an integration by parts using the
formula
$\partial_tg_{ar} = \Gamma_{\ell t}^ng_{rn} + \Gamma_{rt}^ng_{\ell n}$  gives
$$(\Delta Y)^k = -\partial^2_\theta Y^k
-2\Gamma_{\nu\mu}^k\dot\gamma^\nu\partial_\theta Y^\mu -\left
( \partial_\theta\Gamma_{\nu\delta}^k\dot\gamma^\nu
+\Gamma_{\nu\delta}^k\ddot\gamma^\nu +
\Gamma_{\nu\mu}^k\Gamma_{\e\delta}^\mu\dot\gamma^\e\dot\gamma^\nu\right)
Y^\delta.$$
Thus $p_{2s}(\theta, \xi) =  |\xi|^2$ is independent of $\gamma$, but the lower order symbols  depend on 
 derivatives of both $\gamma$ and the metric on $M$. 
 
 We have
 \begin{eqnarray} [D_Z,\weight]X(\gamma)(\theta) &=&
 D_Z \int_{T^*S^1}
 e^{i(\theta-\theta')\cdot \xi} %\tilde 
 p(\theta,\xi) X(\gamma)(\theta')d\theta' d\xi\label{216}\\
 &&\quad - \int_{T^*S^1}
 e^{i(\theta-\theta')\cdot \xi} %\tilde 
 p(\theta,\xi) D_ZX(\gamma)(\theta')d\theta' d\xi. \label{217}
 \end{eqnarray}
 In local coordinates, (\ref{216}) equals
 \begin{eqnarray} \label{218}
 \lefteqn{
\left[ D_Z \int_{T^*S^1}
 e^{i(\theta-\theta')\cdot \xi} %\tilde 
 p(\theta,\xi) X(\gamma)(\theta')d\theta' d\xi\right]^a}\nonumber\\
 &=& \delta_Z\left[
 \int_{T^*S^1}
 e^{i(\theta-\theta')\cdot \xi} %\tilde 
 p(\theta,\xi) X(\gamma)(\theta')d\theta' d\xi\right]^a(\theta)\\
 &&\quad + \Gamma^a_{bc} Z^b(\gamma)(\theta) 
 \left[ \int_{T^*S^1}
 e^{i(\theta-\theta')\cdot \xi} %\tilde 
 p(\theta,\xi) X(\gamma)(\theta')d\theta' d\xi\right]^c(\theta).\nonumber
 \end{eqnarray}
 Here we have suppressed matrix indices in $p$ and $X$.
We can bring $\delta_Z$ past the integral on the right hand side of (\ref{218}).  If
$\gamma_\epsilon$ is a family of curves with $\gamma_0 = \gamma, \dot\gamma_\epsilon = Z$, then
$$\delta_Zp(\theta, \xi) = \frac{d}{d\epsilon}\biggl|_{_{_{\epsilon=0}}} p(\gamma_\epsilon,
\theta,\xi) = \frac{d\gamma_\epsilon^k}{d\epsilon}\biggl|_{_{_{\epsilon=0}}}
\partial_k p(\gamma,\theta, \xi) = Z^k(\gamma(\theta))\cdot \partial_k p(\theta,\xi).$$ Substituting this into (\ref{218}) gives
\begin{eqnarray}\label{219}
\lefteqn{
\left[ D_Z \int_{T^*S^1}
 e^{i(\theta-\theta')\cdot \xi} %\tilde 
 p(\theta,\xi) X(\gamma)(\theta')d\theta' d\xi\right]^a}\nonumber\\
&=& 
\lefteqn{
\left[ \int_{T^*S^1}
 e^{i(\theta-\theta')\cdot \xi} %\tilde 
 Z^k(\gamma)(\theta)\cdot \partial_k p(\theta,\xi) X(\gamma)(\theta')d\theta' d\xi\right]^a}\\
 &&\quad 
+ \Gamma^a_{bc} Z^b(\gamma)(\theta) 
 \left[ \int_{T^*S^1}
 e^{i(\theta-\theta')\cdot \xi} %\tilde 
 p(\theta,\xi) X(\gamma)(\theta')d\theta' d\xi\right]^c(\theta).\nonumber\\
&&\qquad + 
 \left[ \int_{T^*S^1}
 e^{i(\theta-\theta')\cdot \xi} %\tilde 
 p(\theta,\xi) \delta_Z X(\gamma)(\theta')d\theta' d\xi\right]^a(\theta).\nonumber
\end{eqnarray}
Similarly, (\ref{217}) equals
\begin{eqnarray}\label{220}
\lefteqn{\left[\int_{T^*S^1}
 e^{i(\theta-\theta')\cdot \xi} %\tilde 
 p(\theta,\xi) D_ZX(\gamma)(\theta')d\theta' d\xi\right]^a}\nonumber\\
 &=& \left[\int_{T^*S^1}
 e^{i(\theta-\theta')\cdot \xi} %\tilde 
 p(\theta,\xi) \delta_ZX(\gamma)(\theta')d\theta' d\xi\right]^a\\
 &&\quad + 
 \int_{T^*S^1}
 e^{i(\theta-\theta')\cdot \xi} %\tilde 
 p(\theta,\xi)^a_e\Gamma^e_{bc}Z^b(\gamma)(\theta') X^c(\gamma)(\theta')d\theta' d\xi.
 \nonumber
 \end{eqnarray}
 Substituting (\ref{219}), (\ref{220}), into (\ref{216}), (\ref{217}), respectively, gives
 \begin{eqnarray}\label{221}
\lefteqn{( [D_Z,\weight]X(\theta))^a}\\
&=&
Z^b(\theta)\cdot\left[\int_{T^*S^1} e^{i(\theta-\theta')\cdot\xi}\left(\partial_bp^a_e(\theta,\xi) +
\Gamma_{bc}^a(\gamma(\theta)p_e^c(\theta, \xi)\right)X^e(\theta')d\theta'd\xi\right]\nonumber\\
&&\quad - \int_{T^*S^1} e^{i(\theta-\theta')\cdot\xi}p(\theta,\xi)^a_e\Gamma_{bc}^e
(\gamma(\theta'))Z^b(\theta') X^c(\theta')d\theta' d\xi,\nonumber
  \end{eqnarray}
 where $X(\theta') = X(\gamma)(\theta)$ and similarly for Z.
 
 The first term on the right hand side of (\ref{221}) is order zero in $Z$; note that
 $0<2s-1$, since $s>\frac{1}{2}$.  For the last term in (\ref{221}), we do a change of variables typically used in the proof that the composition of $\pdo$s is a $\pdo.$  Set
 \begin{equation}\label{221a}q(\theta, \theta', \xi)^a_b = p(\theta,\xi)^a_e \Gamma_{bc}^e(\gamma(\theta'))X^c
 (\theta'),
 \end{equation}
 so the last term equals
 \begin{eqnarray*}
(PZ)^a(\theta) &\stackrel{\rm def}{=}&  \int_{T^*S^1} e^{i(\theta-\theta')\cdot\xi}q(\theta, \theta', \xi)^a_b Z^b(\theta') d\theta' d\xi\\
 &=& \int_{T^*S^1} e^{i(\theta-\theta')\cdot\xi}q(\theta, \theta', \xi)^a_b e^{i(\theta'-\theta'')
 \cdot\eta} Z^b(\theta'') d\theta'' d\eta \ d\theta' d\xi,
 \end{eqnarray*}
 by applying  Fourier transform and its inverse to $Z$. A little algebra gives
 \begin{equation}\label{222}
 (PZ)^a(\theta) = \int_{T^*S^1} e^{i(\theta-\theta')\cdot\eta}r(\theta,\eta)^a_b Z^b(\theta')
 d\theta' d\eta,
 \end{equation}
 with 
 \begin{eqnarray*}r(\theta, \eta) &=& \int_{T^*S^1} e^{i(\theta-\theta')\cdot(\xi-\eta)}
 q(\theta,\theta', \xi) d\theta' d\xi\\
 &=&  \int_{T^*S^1} e^{it\cdot\xi} q(\theta,\theta - t, \eta + \xi) dt\  d\xi.
 \end{eqnarray*} 
 In the last line we continue to abuse notation by treating the integral in local coordinates in 
 $t = \theta-\theta'$ lying in an interval $I\subset \R$ and implicitly
 summing over a cover and partition of unity of $S^1;$ thus we can consider $q$ as a compactly supported function in $t\in\R.$
 Substituting in the Taylor expansion of $q(\theta,\theta - t, \eta + \xi)$ in $\xi$ gives in local coordinates
  \begin{eqnarray}\label{223a}
  r(\theta, \eta) &=& \int_{T^*\R} e^{it\cdot \xi} \left[ \sum_{\alpha, |\alpha|=0}^N
 \frac{1}{\alpha!} \partial_\xi^\alpha|_{\xi=0} q(\theta,\theta-t, \eta+\xi)\xi^\alpha + {\rm O}
 (|\xi|^{N+1})\right] dt \  d\xi\nonumber\\
 &=& \sum_{\alpha, |\alpha|=0}^N \frac{i^{|\alpha|}}{\alpha!} \partial^\alpha_t\partial^\alpha
 _\xi q(\theta, \theta, \eta) + {\rm O} (|\xi|^{N+1}).
 \end{eqnarray}
 Thus $P$ in (\ref{222}) is a $\pdo$ with apparent top order symbol 
 $q(\theta, \theta, \eta)$, which by (\ref{221a}) has order $2s.$  The top order symbol can be computed in any local coordinates on $S^1$ and $\gamma^*TM$.  If we choose
 manifold coordinates (see \S2.3) which are 
Riemannian normal coordinates centered at $\gamma(\theta)$, the Christoffel symbols vanish at this point, and
 so
 $$q(\theta, \theta, \eta)^a_b = p(\theta,\xi)^a_e\Gamma_{bc}^e(\gamma(\theta)) X^c(\theta)
 =0.$$
 Thus $P$ is in fact of order $2s-1$, and so both terms on the right hand side of (\ref{221}) have order at most $2s-1$.

\end{proof}

\begin{rem} (i) For $s\in\Z^+$,  $\delta_Z\weight$ differs
from the usual definition by a smoothing operator.  

(ii) For all $s$, the proof of Lemma \ref{pdo}(i) shows that for all $k$, 
$\sigma_k(\delta_Z\weight) = \delta_Z(\sigma_k(\weight))$ in local coordinates.

\end{rem}

We can now complete the computation of the Levi-Civita connection for general $s.$

Let $[D_\cdot,\weight]X^*$ be the formal $L^2$ adjoint  of $[D_\cdot,\weight]X$.
We abbreviate $[D_\cdot,\weight]X^*(Y)$ by $[D_Y,\weight]X^*.$

\begin{thm}  \label{thm25} (i) For $s>\frac{1}{2}$, 
The Levi-Civita connection for the $H^s$ metric is given by
\begin{eqnarray}\label{quick}\nabla_X^sY &=& D_XY + \frac{1}{2}\wgtsi[D_X, \wgts]Y +
  \frac{1}{2}\wgtsi[D_Y, \wgts]X \nonumber\\
&&\quad -\frac{1}{2} \wgtsi[D_Y,\weight]X^*.
\end{eqnarray}

(ii) The connection one-form $\omega^s$ in exponential coordinates is given by
\begin{eqnarray}\label{223}\lefteqn{\omega^s_X(Y)(\gamma) (\theta)}\\
&=& \omega^M(Y)(\gamma(\theta)) + 
\left(\frac{1}{2}\wgtsi[D_X, \wgts]Y +
  \frac{1}{2}\wgtsi[D_Y, \wgts]X \right.\nonumber\\
  &&\quad \left.
-\frac{1}{2} \wgtsi[D_Y,\weight]X^*\right)(\gamma)(\theta).\nonumber
\end{eqnarray}

(iii) The connection one-form takes values in zeroth order $\pdo$s.
\end{thm}

\begin{proof}  Since $[D_Z,\weight]X$ is a $\pdo$ in $Z$ of order $2s-1$, its formal adjoint is
a $\pdo$ of the same order.  Thus
$$\langle [D_Z,\weight]X,Y\rangle_0 = \langle Z, [D_\cdot, \weight]X^*(Y)\rangle
= \langle Z, \wgtsi[D_Y,\weight]X^*\rangle_s.$$
Thus $A_XY$ in (\ref{axy}) satisfies
$A_XY = \wgtsi[D_Y,\weight]X^*.$  Lemma \ref{lem: LCs} applies to all $s>\frac{1}{2}$, 
so (i) follows.  (ii) follows as 
in Corollary \ref{cor2}.  Since $\omega^M$ is zeroth order and all 
other terms have order $-1$, (iii) holds as well.
\end{proof}

\begin{rem}  This theorem implies that the Levi-Civita connection exists for the 
$H^s$ metric in the strong sense:  for $X\in  T_\gamma LM =H^{s'-1}(\gamma^*TM)$
and $Y_\eta\in H^{s'-1}(\eta^*TM)$ a smooth vector field on $LM = H^{s'}(S^1,M)$,
 $\nabla^s_XY(\gamma)\in H^{s'-1}(\gamma^*TM).$  (See Remark
2.1.)  For each term except $D_XY$ on the right hand side of (\ref{quick}) is order
$-1$ in $Y$, and so takes $H^{s'-1}$ to $H^{s'}\subset H^{s'-1}.$  For $D_XY = \delta_XY + \Gamma\cdot Y$, $\Gamma$ is zeroth order and so bounded on $H^{s'-1}.$  Finally, 
the definition of a smooth vector field on $LM$ implies that $\delta_XY$ stays in $H^{s'-1}$
for all $X$.
\end{rem}

\subsection{{\bf Extensions of the Frame Bundle of $LM$}}\label{extframe}

In this subsection we discuss the choice of structure group for the
$H^s$ and Levi-Civita connections on $LM.$

Let $\calH$ be the Hilbert space
 $H^{s_0}(\gamma^*TM)$ for 
a fixed $s_0$ and $\gamma.$ 
 Let $GL(\calH)$ be the group of bounded invertible linear
operators on $\calH$; inverses of elements are bounded by the closed graph
theorem.   $GL(\calH)$ has the subset
topology of the norm topology on ${\mathcal B}(\calH)$, the bounded linear
operators on $\calH$.
$GL(\calH)$ is an infinite dimensional Banach Lie group, as a group which
is an open subset of the infinite dimensional Hilbert manifold 
${\mathcal B}(\calH)$
\cite[p.~59]{Omori}, and has Lie algebra 
${\mathcal B}(\calH)$. Let $\pdo_{\leq 0}, 
\pdo_0^*$ denote the algebra of classical
$\pdo$s of nonpositive order 
and the group of invertible zeroth order $\pdo$s, respectively,
where all $\pdo$s act on $\calH.$   
Note that $\pdo_0^*\subset GL(\calH).$
   
\begin{rem} 
The inclusions of $\pdo_0^*, \pdo_{\leq 0}$ into $GL(\calH), {\mathcal
    B}(\calH)$ are trivially continuous in the subset topology.
For the Fr\'echet topology on $\pdo_{\leq 0}$, 
the  inclusion is 
continuous as in \cite{lrst}.
\end{rem}

We recall
the relationship between
 the  connection one-form $\alpha_{FN}$ on the frame bundle $FN$ of a
 manifold $N$
and
local expressions for the connection on $TN.$ For $U\subset N$,
 let $\chi:U\to FN$ be a local section.
 A metric  connection $\nabla$ on $TN$ with local
connection one-form $\omega$ determines a connection $\alpha_{FN}\in
 \Lambda^1(FN, {\mathfrak o}(n))$ on $FN$
by {\it (i)} $\alpha_{FN}$ is the Maurer-Cartan one-form on each fiber,
and {\it (ii) }
$\alpha_{FN}(Y_u)=\omega (X_p),$ for $ Y_u=\chi_*X_p$
\cite[Ch.~8, Vol.~II]{Spi}, or equivalently
%\begin{equation}\label{localexp}
$\chi^*\alpha_{FN} = \omega.$
%\end{equation}

This applies to $N=LM.$
The frame bundle $FLM\to LM$ is constructed
as in the finite dimensional case. The
fiber over $\gamma$ is isomorphic to the gauge group $\calG$ of $\calR$
and fibers are glued by the transition functions for
$TLM$. Thus the frame bundle is
topologically a
$\calG$-bundle.

%%%%%%%%%%%  use both the LC and the H^s connection
However, by Theorem \ref{thm25},
the Levi-Civita connection one-form $\omega^s_X$
takes
values in $\pdo_{\leq 0}$. 
The curvature two-form $\Omega^{s} = d_{LM}\omega^{s} + \omega^{s}\wedge
\omega^s$ also takes values in $\pdo_{\leq 0}.$  (Here $d_{LM}\omega^{s}(X,Y)$
is defined by the Cartan formula for the exterior derivative.)
These
forms should take values in the Lie algebra of the structure
group.  Thus we should extend the structure group to the Fr\'echet Lie group
 $\pdo_0^*$, since its Lie
algebra is $\pdo_{\leq 0}$  \cite{paycha}.
This leads to an extended
frame bundles, also denoted $FLM$. The  transition
 functions are unchanged, since 
$\calG \subset \pdo_0^*$.
 Thus $(FLM,\alpha^s = \alpha^s_{FLM})$ as a geometric
bundle (i.e.~as a  bundle with connection $\alpha^s$ associated to
$\nabla^{s}$) is a $\pdo_0^*$-bundle.

In summary, for the Levi-Civita connections we have
$$ \begin{array}{ccc}
\calG&\longrightarrow &FLM\\
& & \downarrow\\
& & LM
\end{array}
\ \ \ \ \ \ \ \ \ \ \ \ \ \ 
\begin{array}{ccc}
\pdo_0^*&\longrightarrow &(FLM,\alpha^s)\\
& & \downarrow\\
& & LM
\end{array}
$$

\begin{rem}\label{rem:ext}  If
 we extend the structure group of the frame bundle with
  connection from $\pdo_0^*$ to
  $GL(\calH)$, the frame bundle becomes trivial by Kuiper's theorem.  
  %This
%would allow us to define Chern-Simons forms for the trivial connection on $LM$ by the
%  procedures of \S5.4.
Thus
there is a potential loss of information if 
we pass to the larger frame
  bundle.

The situation is similar to the following examples.  Let $E\to S^1$ be
the $GL(1,\R)$ (real line)
 bundle with gluing functions (multiplication by) $1$ at $1\in
S^1$ and $2$ at $-1\in S^1.$  $E$ is trivial as a $GL(1,\R)$-bundle, 
with global section $f$ with $\lim_{\theta\to -\pi^+}f(e^{i\theta}) = 1, 
f(1) = 1,
\lim_{\theta\to\pi^-}f(e^{i\theta}) = 1/2.$  
However, as a $GL(1,\Q)^+$-bundle, $E$ is nontrivial, as a
global section is locally constant. As a second example,
 let $E\to M$ be a nontrivial
$GL(n,\C)$-bundle. Embed $\C^n$ into a Hilbert space $\calH$, and extend $E$
to an $GL(\calH)$-bundle $\calE$ 
with fiber $\calH$ and with the  transition functions for $E$ (extended by the identity in
directions perpendicular to the image of $E$).  Then $\calE$ is
trivial.

\end{rem}

\section{{\bf The Loop Group Case}}

In this section, we relate our work to Freed's work on based loop groups
$\Omega G$
\cite{Freed}.  We find a particular representation of the loop algebra that
controls 
the order of the curvature of the $H^1$ metric on $\Omega G.$

$\Omega G\subset LG$ has tangent space $T_\gamma\Omega G
= \{X\in T_\gamma LG: X(0) = X(2\pi) = 0\}$ in some Sobolev topology.  
Instead of using 
$D^2/d\gamma^2$ to define the Sobolev spaces, the usual choice is
$\Delta_{S^1} = -d^2/d\theta^2$ coupled to the
identity operator on the Lie algebra ${\mathfrak g}$.  Since this operator has
no kernel on $T_\gamma\Omega M$, 
$1 + \Delta$ is replaced by
 $\Delta$.  These changes in the $H^s$ inner product
do not alter the spaces of Sobolev sections, but the $H^s$ metrics on $\Omega G$ are 
no longer induced from a metric on $G$ as in the previous sections.

This simplifies the calculations of the Levi-Civita connections.
In particular,\\
 $[D_Z,\Delta^s] = 0$, so there is no term $A_XY$ as in (\ref{axy}).  
%The $s$ Levi-Civita connection is given by an explicit $\pdo$ for all $s > 1/2$, in contrast to the results in \S2.5.  
As a result, one can work directly with the six term formula (\ref{5one}).
For $X, Y, Z$ left
invariant vector fields, the first three terms on the right hand side of 
(\ref{5one}) vanish. Under the standing assumption that $G$ has a left
invariant, 
Ad-invariant inner product,
one obtains
$$2\nabla^{(s)}_XY = [X,Y] + \Delta^{-s}[X,\Delta^sY] +
\Delta^{-s}[Y,\Delta^sX]$$
\cite{Freed}.

It is an interesting question to compute the order of the curvature operator
as a function of $s$.  For based loops, Freed proved that this order is at
most $-1$.  In \cite{andres}, it is
shown that the order of $\Omega^s$ is at most $-2$ for all $s\neq 1/2, 1$ on
both $\Omega G$ and $LG$, and is exactly $-2$ for $G$ nonabelian.  
 For the case $s=1$, we have a much stronger result.

\begin{prop} The curvature of the
Levi-Civita connection for the $H^1$ inner product on $\Omega
G$ associated to $-\frac{d^2}{d\theta^2}\otimes {\rm Id}$ is a $\pdo$ of order $-\infty.$
\end{prop}

\noindent {\sc Proof:}
We give two quite different proofs.  

By \cite{Freed}, the $s=1$ curvature operator $\Omega = \Omega^{1}$
satisfies
$$\left\langle \Omega(X,Y)Z,W\right\rangle_1 = \left(\int_{S^1}[Y,\dot
Z],\int_{S^1}[X,\dot W]\right)_{\mathfrak g} - (X\leftrightarrow Y),$$
where  the inner product is the Ad-invariant form on the Lie algebra ${\mathfrak g}$.  We
want to write the right hand side of 
this equation as an $H^1$ inner product with $W$, in
order to recognize $\Omega(X,Y)$ as a $\pdo.$

Let $\{e_i\}$ be an orthonormal basis of ${\mathfrak g}$, considered
as a left-invariant frame of $TG$ and as global sections of $\gamma^*TG.$
 Let $\cc{i}{j}{k} = ([e_i,e_j],
e_k)_{\mathfrak g}$ be the structure constants of ${\mathfrak g}.$
(The Levi-Civita connection on left invariant vector fields
for the left invariant metric is
given by $\nabla_XY = \frac{1}{2}[X,Y]$, so the structure constants
are twice the Christoffel symbols.)  For $X = X^ie_i =
X^i(\theta)e_i, Y = Y^je_j,$ etc., integration by parts 
gives
$$\left\langle\Omega(X,Y)Z,W\right\rangle_1 = \left(\int_{S^1} \dot
Y^iZ^jd\theta\right)\left( \int_{S^1}\dot X^\ell W^m d\theta\right)
\cc{i}{j}{k}\cc{\ell}{m}{n}\delta_{kn} - (X\leftrightarrow Y).$$
Since
$$\int_{S^1}\cc{\ell}{m}{n}\dot X^\ell W^m =
\int_{S^1}\left(\delta^{mc}\cc{\ell}{c}{n}\dot X^\ell
e_m,W^be_b\right)_{\mathfrak g} = \left
\langle \Delta^{-1}(\delta^{mc}\cc{\ell}{c}{n} \dot X^\ell e_m),
W\right\rangle_1,$$
we get
\begin{eqnarray*}
\langle\Omega(X,Y)Z,W\rangle_1 &=& \left\langle
 \left[\int_{S^1} \dot Y^i Z^j\right]
\cc{i}{j}{k}\delta_{kn}\delta^{ms}\cc{\ell}{s}{n} \Delta^{-1}(\dot
X^\ell e_m),W\right\rangle_1- (X\leftrightarrow Y)\\
&=&\left\langle \left[ \int_{S^1}
a_j^k(\theta,\theta')Z^j(\theta')d\theta'\right] e_k,W\right\rangle_1,
\end{eqnarray*}
with
\begin{equation}\label{a}a_j^k(\theta,\theta') = \dot Y^i(\theta')
\cc{i}{j}{r}\delta_{rn}\delta^{ms}\cc{\ell}{s}{n} 
%%\cc{i}{j}{s}\delta_{sn} 
\left( \Delta
^{-1}( \dot X^\ell
e_m)\right)^k(\theta) - (X\leftrightarrow Y).
\end{equation}

We now show that $Z\mapsto \left(\int_{S^1}
a_j^k(\theta,\theta')Z^j(\theta')d\theta'\right)e_k$ is a smoothing
operator.  Applying Fourier transform and Fourier inversion to $Z^j$
yields
\begin{eqnarray*} \int_{S^1} a_j^k(\theta,\theta')Z^j(\theta')d\theta'
&=& \int_{S^1\times\R\times S^1}
a_j^k(\theta,\theta')e^{i(\theta'
-\theta'')\cdot\xi}Z^j(\theta'')d\theta''d\xi d\theta'\\
&=&
\int_{S^1\times\R\times S^1} \left[ a_j^k(\theta,\theta')e^{-i(\theta
-\theta')\cdot\xi}\right]e^{i(\theta
-\theta'')\cdot\xi}Z^j(\theta'')d\theta''d\xi d\theta',
\end{eqnarray*}
so $\Omega(X,Y)$ is a $\pdo $ with symbol 
\begin{equation}\label{b} b_j^k(\theta,\xi) =
\int_{S^1} a_j^k(\theta,\theta') \eee d\theta',
\end{equation}
with the usual mixing of local and global notation.

For fixed $\theta$,
(\ref{b})  contains the Fourier transform of $\dot Y^i(\theta')$  and $\dot X^i(\theta')$, as
these are the only $\theta'$-dependent terms in (\ref{a}).
%pieces of the two terms in $a_j^k(\theta,\theta')$ which depend  on $\theta'.$ 
Since the
Fourier transform is taken in a local chart with respect to a
partition of unity, and since in each chart $\dot Y^i$ and $\dot X^i$ times the
partition of unity function is compactly supported, the Fourier
transform of $a_j^k$ in each chart is rapidly decreasing.  Thus
$b_j^k(\theta,\xi)$ is the product of a rapidly decreasing function
with $e^{i\theta\cdot\xi}$, and hence is of order $-\infty.$

We now give a second proof.  For all $s$,
$$\nabla_X Y = \frac{1}{2}[X,Y] -\frac{1}{2} \Delta^{-s}[\Delta^sX,Y]
+\frac{1}{2}\Delta^{-s}[X,\Delta^sY].$$
Label the terms on the right hand side (1) -- (3).
 As an operator on $Y$ for fixed $X$, the symbol of (1) is
$\sigma((1))^a_\mu = \frac{1}{2}X^ec_{\e\mu}^a.$
 Abbreviating $\xii{-s}$ by $\xi^{-2s}$, we have
\begin{eqnarray*} \sigma((2))^a_\mu &\sim & -\frac{1}{2}c_{\e\mu}^a
\left[ \xi^{-2s}\Delta^sX^\e -\frac{2s}{i}\xi^{-2s-1}
\partial_\theta\Delta^s X^\e  \right.\\
&&\ \ \  \left. +\sum_{\ell=2}^\infty\frac{(-2s)(-2s-1)
\ldots(-2s-\ell+1)}{i^\ell \ell!}\xi^{-2s-\ell}
\partial_\theta^\ell\Delta^s X^\e \right]\\
\sigma((3))^a_\mu &\sim &  \frac{1}{2}c_{\e\mu}^a
\left[ X^\e+ \sum_{\ell=1}^\infty \frac{(-2s)(-2s-1)
\ldots(-2s-\ell+1)}{i^\ell \ell!} \xi^{-\ell}\partial_\theta^\ell X^\e\right].
\end{eqnarray*}
Thus
\begin{eqnarray}\label{fourone}
\sigma(\nabla_X)^a_\mu  &\sim& \frac{1}{2}c_{\e\mu}^a\left[ 2X^\e
   -\xi^{-2s}\Delta^sX^\e
%-\frac{1}{i}
+\frac{2s}{i}
\xi^{-2s-1}\partial_\theta\Delta^sX^\e\right. \nonumber\\
&&\ \ \
   -\sum_{ \ell=2}^\infty\frac{(-2s)(-2s-1)\ldots(-2s-\ell+1)}{i^\ell \ell!}
\xi^{-2s-\ell}\partial_\theta^\ell\Delta^s X^\e \\
&&\ \ \  \left. + \sum_{\ell=1}^\infty \frac{(-2s)(-2s-1)
\ldots(-2s-\ell+1)}{i^\ell \ell!} \xi^{-\ell}\partial_\theta^\ell
 X^\e. \right].\nonumber
\end{eqnarray}

Set $s=1$ in (\ref{fourone}), and replace $\ell$
by
$\ell-2$ in the first infinite sum.  Since $\Delta = -\partial_\theta^2$, a
little algebra gives
\begin{equation}\label{fourtwo}
\sigma(\nabla_X)^a_\mu \sim c_{\e\mu}^a\sum_{\ell=0}^\infty
\frac{(-1)^\ell}{i^\ell}
\partial_\theta^\ell X^\e\xi^{-\ell}
=  \ad\left( \sum_{\ell=0}^\infty
\frac{(-1)^\ell}{i^\ell}\partial_\theta^\ell
X\xi^{-\ell}
%e_k
\right).
\end{equation}
%   where $\{e_k\}$ is an orthonormal basis of ${\mathfrak g}$ as before.

Denote the infinite sum in the last term of (\ref{fourtwo})
by $W(X,\theta,\xi)$. The map
$X\mapsto W(X,\theta,\xi)$ takes the  Lie algebra of left invariant vector
fields on $LG$ to the Lie algebra
$L{\mathfrak g}[[\xi^{-1}]], $
the space of formal $\pdo$s of nonpositive integer order on the trivial bundle
$S^1\times{\mathfrak g} \to S^1$, where the Lie bracket on the
target involves multiplication of power series and bracketing in
${\mathfrak g}.$  We claim that this map is a Lie algebra homomorphism.
Assuming this, we see that
\begin{eqnarray*} \sigma\left(\Omega(X,Y)\right) &=&
  \sigma\left([\nabla_X,\nabla_Y] -\nabla_{[X,Y]}\right)
\sim \sigma\left( [\ad W(X), \ad W(Y)] - \ad W([X,Y]) \right)\\
&=& \sigma\left( \ad ( [W(X), W(Y)]) - \ad W([X,Y]) \right) = 0,
\end{eqnarray*}
which proves that $\Omega(X,Y)$ is a smoothing operator.

To prove the claim,
set $X = x^a_n\eff{n}e_a, Y =y^b_m\eff{m}e_b$. %  for simplicity.  
Then
\begin{eqnarray*} W([X,Y]) &=&
 W( x^ny^m\eff{(n+m)}c_{ab}^k e_k) =\sum_{\ell=0}^\infty \frac{(-1)^\ell}
{i^\ell } c_{ab}^k \partial_\theta^\ell
 \left(x^a_ny^b_m\eff{(n+m)}\right) \xi^{-\ell}e_k\\
  {[} W(X)  ,   W(Y)]
&=& \sum_{\ell=0}^\infty \sum_{p+q = \ell}
\frac{(-1)^{p+q}}{i^{p+q}} \partial_\theta^p \left(
x^a_n\eff{n}\right) \partial_\theta^q
\left( y^b_m\eff{m}\right)\xi^{-(p+q)}c_{ab}^k e_k,
\end{eqnarray*}
and these two sums are clearly equal.
\hfill $\Box$

\bigskip

It would be interesting to understand how the map $W$ fits into the
representation theory of the loop algebra $L{\calg}.$

\appendix
\section{{\bf Local Symbol Calculations}}\label{localsymbols}

We  compute the $0$ and $-1$ order symbols of the
connection one-form $\omega^1$ and the curvature two-form $\Omega^1$ of
 the $s=1$ Levi-Civita connection. 
 We also compute the $0$ and $-1$ order symbols of the
connection one-form for the general $s>\frac{1}{2}$ connection, and the $0$ order symbol of the 
curvature of the general $s$ connection.
 The formulas  show that
the $s$-dependence of these symbols is
linear, which will be used to define regularized (i.e. $s$-independent) Wodzicki-Chern-Simons classes
in \cite{MRT2}.

\subsection{{\bf Connection and Curvature Symbols for $s=1$}}
${}$
\medskip

%\omega^1In this subsection $\omega = \omega^1, \Omega = \Omega^1.$

Using Corollary \ref{cor2}, we can compute these symbols easily in manifold coordinates. 

\begin{lem} \label{old2.1}
(i) At $\gamma(\theta)$,
$\sigma_0(\omega^1_X)^a_b =  (\omega^M_X)^a_b = \chw{c}{b}{a}X^c.$

(ii)  \begin{eqnarray*}
\frac{1}{i|\xi|^{-2}\xi}\sigma_{-1}(\omega^1_X) &=& \frac{1}{2}(-2R(X,\dg)
-R(\cdot,\dg)X + R(X,\cdot)\dg).
\end{eqnarray*}
Equivalently,
\begin{eqnarray*}
\frac{1}{i|\xi|^{-2}\xi}\sigma_{-1}(\omega^1_X)^a_b &=& 
\frac{1}{2}
(-2R_{cdb}^{\ \ \ a} -R_{bdc}^{\ \ \ a} + R_{cbd}^{\ \ \  a})X^c\dg^d.
\end{eqnarray*}
\end{lem}

\begin{proof} (i) For $\sigma_0(\omega^1_X)$, the only term in (\ref{two}) of order
  zero is the Christoffel term.  

(ii) For $\sigma_{-1}(\omega^1_X)$, label the last six terms on the right hand side of
(\ref{two}) by (a), ..., (f).  By Leibniz rule for the tensors, the only
terms of order $-1$ come from:
in (a), $-\nabla_{\dg}(R(X,\dg)Y) = -R(X, \dg) \nabla_{\dg}Y +$ lower order in
$Y$;
in (b), the term  $-R(X, \dg) \nabla_{\dg}Y$;
in (c), the term $-R(\nabla_{\dg}Y, \dg)X$;
in (e), the term $R(X,\nabla_{\dg}Y)\dg.$

For any vectors $Z, W$, the curvature endomorphism $R(Z, W): TM\to TM$ has
$$R(Z,W)^a_b = R_{c d b}^{\ \ \ a}Z^cW^d.$$
  Also, since $(\nabla_{\dg}Y)^a =
\frac{d}{d\theta}Y^a $ plus zeroth order terms,
$\sigma_{1}(\nabla_{\dg}) =
i\xi\cdot {\rm Id}.$
Thus in (a) and (b), 
$\sigma_1(-R(X, \dg) \nabla_{\dg})^a_b = -R_{cdb}^{\ \ \ a}X^c\dg^d\xi.$

For (c), we have $-R(\nabla_{\dg}Y, \dg)X = -R_{cdb}^{\ \ \
  a}(\nabla_{\dg}Y)^c\dg^d X^b\partial_a$, so the top order symbol is
$-R_{cdb}^{\ \ \ a}\xi\dg^dX^b = -R_{bdc}^{\ \ \ a}\xi\dg^d X^c.$

For (e), we have $R(X,\nabla_{\dg}Y)\dg = R_{cdb}^{\ \ \ a}X^c(\nabla_{\dg}Y)^d
\dg^b\partial_a$, so the top order symbol is 
$R_{cdb}^{\ \ \ a}X^c\xi \dg^b = R_{cbd}^{\ \ \ a}X^c\xi \dg^d.$

Since the top order symbol of $\wgti$ is $|\xi|^{-2}$, adding these four terms
finishes the proof. 
\end{proof}
 
 We now compute the top symbols of the curvature tensor.  $\sigma_{-1}(\Omega^1)$ involves
 the covariant derivative of the curvature tensor on $M$, but fortunately this symbol
 will not be needed in \cite{MRT2}.

\begin{lem}\label{old2.2}
(i) 
$\sigma_0(\Omega^1(X,Y))^a_b =  R^M(X,Y)^a_b = R_{cdb}^{\ \ \ a}X^cY^d.$

(ii) \begin{eqnarray*}
\frac{1}{i|\xi|^{-2}\xi}\sigma_{-1}(\Omega^1(X,Y)) &=&
\frac{1}{2}\left(\nabla_X[-2R(Y,\dg) - R(\cdot,\dg)Y +
  R(Y,\cdot)\dg]\right.\\
&&\qquad \left. - \xly \right.\\
%-\nabla_Y[-2R(X,\dg) - R(\cdot,\dg)X +
%  R(X,\cdot)\dg]\right\}.
&&\qquad\left. - [-2R([X,Y],\dg) -R(\cdot,\dg)[X,Y] + R([X,Y],\cdot)\dg] \right).
\end{eqnarray*}
Equivalently, in Riemannian normal coordinates on $M$ centered at $\gamma(\theta)$,
\begin{eqnarray}\label{moc}
\frac{1}{i|\xi|^{-2}\xi}\sigma_{-1}(\Omega^1(X,Y))^a_b &=& \frac{1}{2}
X[(-2R_{cdb}^{\ \ \ a} -R_{bdc}^{\ \ \ a} + R_{cbd}^{\ \ \
  a})\dg^d]Y^c - (X\leftrightarrow Y)\nonumber\\
&=&\frac{1}{2}
X[-2R_{cdb}^{\ \ \ a} -R_{bdc}^{\ \ \ a} + R_{cbd}^{\ \ \
  a}]\dg^dY^c  -(X\leftrightarrow Y)\\
&&\qquad +
\frac{1}{2}[-2R_{cdb}^{\ \ \ a} -R_{bdc}^{\ \ \ a} + R_{cbd}^{\ \ \
  a}]\dot X^dY^c - (X\leftrightarrow Y)\nonumber
\end{eqnarray}
\end{lem}

\begin{proof} 
(i)
\begin{eqnarray*} \sigma_0(\Omega^1(X,Y))^a_b &=& \sigma_0((d\omega^1 +
  \omega^1\wedge\omega^1)(X,Y))^a_b\\
&=& [(d\sigma_0(\omega^1) + \sigma_0(\omega^1)\wedge\sigma_0(\omega^1))(X,Y)]^a_b\\
&=& [(d\omega^M + \omega^M\wedge\omega^M)(X,Y)]^a_b\\
&=& R^M(X,Y)^a_b = R_{cdb}^{\ \ \ a}X^cY^d.
\end{eqnarray*}

(ii) Since $\sigma_0(\Omega^1_X)$ is independent of $\xi$, after dividing by
$i|\xi|^{-2}\xi$ we have
\begin{eqnarray*}\sigma_{-1}(\Omega^1(X,Y))^a_b &=& (d\sigma_{-1}(\omega^1)
  (X,Y))^a_b + \sigma_0(\omega^1_X)^a_c\sigma_{-1}(\omega^1_Y)^c_b
+ \sigma_{-1}(\omega^1_X)^a_c\sigma_{0}(\omega^1_Y)^c_b\\
&&\qquad
-\sigma_0(\omega^1_Y)^a_c\sigma_{-1}(\omega^1_X)^c_b
+ \sigma_{-1}(\omega^1_Y)^a_c\sigma_{0}(\omega^1_X)^c_b.
\end{eqnarray*}
As an operator on sections of $\gamma^*TM$, 
$\Omega^{LM} - \Omega^M$ has order $-1$  so $\sigma_{-1}(\Omega^{LM})
= \sigma_{-1}(\Omega^{LM} -\Omega^M)$ is independent of coordinates.
In Riemannian normal coordinates at $\gamma(\theta)$, $\sigma_0(\omega_X) = \sigma_0(\omega_Y) = 0$, so
\begin{eqnarray*}\sigma_{-1}(\Omega^1(X,Y))^a_b &=&
  (d\sigma_{-1}(\omega^1)(X,Y))^a_b\\
&=& X(\sigma_{-1}(\omega^1_Y))^a_b - Y(\sigma_{-1}(\omega^1_X))^a_b
  -\sigma_{-1}(\omega^1_{[X.Y]})^a_b\\
&=& \frac{1}{2} X[(-2R_{cdb}^{\ \ \ a} -R_{bdc}^{\ \ \ a} + R_{cbd}^{\ \ \
  a}]Y^c\dg^d] - (X\leftrightarrow Y)\\
&&\qquad -\frac{1}{2}( -2R_{cdb}^{\ \ \ a} -R_{bdc}^{\ \ \ a} + R_{cbd}^{\ \ \
  a}][X,Y]^c\dg^d.
\end{eqnarray*}
The terms involving $X(Y^c) - Y(X^c) - [X,Y]^c$ cancel (as they must, since the symbol two-form 
cannot involve derivatives of $X$ or $Y$).  Thus
$$\sigma_{-1}(\Omega^1(X,Y))^a_b = \frac{1}{2} X[(-2R_{cdb}^{\ \ \ a} 
-R_{bdc}^{\ \ \ a} + R_{cbd}^{\ \ \  a})Y^c\dg^d] - (X\leftrightarrow Y).$$

This gives the first coordinate expression in (\ref{moc}). The second expression follows from
 $X(\dg^d) = \dot X^d $ (see (\ref{badterms})).

To convert from the coordinate expression to the covariant expression, we follow the 
usual procedure of changing ordinary derivatives to covariant derivatives and adding bracket terms.  For example,
\begin{eqnarray*}\nabla_X(R(Y,\dg)) &=& (\nabla_XR)(Y,\dg) + R(\nabla_XY,\dg)
  + R(Y,\nabla_X\dg) \\
&=& X^iR_{cdb\ ;i}^{\ \ \ a}Y^c\dg^d + R(\nabla_XY,\dg) + R_{cdb}^{\ \ \
    a}Y^c(\nabla_X\dg)^d.
%&=& X^iR_{cdb\ ;i}^{\ \ \ a}Y^c\dg^d + R(\nabla_XY,\dg) + R_{cdb}^{\ \ \
%    a}Y^c\dot X^d.
\end{eqnarray*}
In Riemannian normal coordinates at
$\gamma(\theta)$, we have $X^iR_{cdb\ ;i}^{\ \ \ a} = X^i\partial_i R_{cdb}^{\
  \ \ a} = X(R_{cdb}^{\ \ \ a})$ and $(\nabla_X\dg)^d =  X(\dg^d).$ 
  % = \dot X^d.$ 
  Thus 
  $$\nabla_X(R(Y,\dg)) -\xly - R([X,Y],\dg) = X(R_{cdb}^{\ \ \ a}\dg^d)Y^c - \xly.$$
  The other terms are handled similarly.
  \end{proof} 

\subsection{{\bf Connection and Curvature Symbols for General $s$}}
${}$
\medskip

The noteworthy feature of these computations is the linear dependence of $\sigma_{-1}(\omega^{s})$ on $s$.  

Let $g$ be the Riemannian metric on $M$.

\begin{lem}\label{lem:33}
(i) At $\gamma(\theta)$,
$\sigma_0(\omega^s_X)^a_b =  (\omega^M_X)^a_b = \chw{c}{b}{a}X^c.$

(ii) $\sigma_0(\Omega^s(X,Y))^a_b =  R^M(X,Y)^a_b = R_{cdb}^{\ \ \ a}X^cY^d.$

(iii)  $
\frac{1}{i|\xi|^{-2}\xi}\sigma_{-1}(\omega^s_X)^a_b = s T(X,\dg, g)$,
where $T(X, \dg, g)$  is tensorial and independent of $s$. 
\end{lem}

\begin{proof} (i)  By Lemma \ref{pdo}, the only term of order zero in (\ref{223}) 
 is $\omega^M_X.$

(ii)  The proof of Lemma \ref{old2.2}(ii) carries over.  

(iii) By Theorem \ref{thm25}, we have to compute $\sigma_{2s-1}$ for $[D_X,\wgts]$, 
$[D_\cdot,\wgts]X$, and $[D_\cdot,\wgts]X^*$, as 
$\sigma_{-1} (\wgtsi[D_X,\wgts]) =  |\xi|^{-2s}\sigma_{-1}([D_X,\wgts])$, etc.

Write $D_X = \delta_X + \Gamma\cdot X$ in shorthand.  Since $\wgts$ has scalar leading order symbol, $[\Gamma\cdot X,\wgts]$ has order $2s-1.$  Thus
we can compute $\sigma_{2s-1}([\Gamma\cdot X,\wgts])$ in any coordinate system.  
In  Riemannian normal coordinates centered at $\gamma(\theta)$, as in the proof of Lemma \ref{pdo}(ii),  the Christoffel symbols vanish.
Thus $\sigma_{2s-1}([\Gamma\cdot X,\wgts]) =0.$

By (\ref{tsmo}), $\sigma_{2s-1}([\delta_X,\wgts])$ is $s$ times a tensorial expression in $X, \dg, g$,
since $\partial_i\Gamma_{\nu j}^{\ell} = \frac{1}{3}(R_{i\nu j}^{\ \ \ \ell} +
R_{ij\nu }^{\ \ \ \ell})$ in normal coordinates.  The term with $\Gamma$ vanishes, so 
$\sigma_{2s-1}( [D_X,\wgts]) $ is $s$ times this tensorial expression.

The argument for $\sigma_{2s-1}([D_\cdot,\wgts]X$ is similar.  The term
with 
$\Gamma  $ vanishes.  %In the notation of Lemma \ref{pdo}, 
By (\ref{222}), (\ref{223a}), 
$$\sigma_{2s-1}([\delta_\cdot,\wgts]X)^a_b = 
i\sum_j\partial_t^j\partial_\xi^j|_{t=0, \xi=0} (p(\theta, \xi)^a_e\Gamma_{bc}^e(\gamma-t, \eta +\xi) X^c(\theta-t)).$$
By (\ref{tsmo}), the right hand side is linear in $s$ for Re$(s) <0$.  By (\ref{abc}), this implies 
the linearity in $s$ for Re$(s)>0.$  

Since $\sigma_{2s-1}([D_\cdot,\weight]X^*) = (\sigma_{2s-1}([D_\cdot, \weight]X))^*$, this 
symbol is also linear in $s$.
\end{proof}

\bibliographystyle{amsplain}
\bibliography{Paper3}

\end{document}